\documentclass[a4paper,11pt,reqno,twoside]{amsart}

\usepackage{amsmath}
\usepackage{amsfonts}
\usepackage{amssymb}
\usepackage{amsthm}
\usepackage{mathtools}
\usepackage{color}
\usepackage{ifpdf}
\usepackage{array}
\usepackage{url}
\usepackage{multirow,bigdelim}
\usepackage{ascmac}
\usepackage[all]{xy} 
\usepackage{graphicx}

\newtheorem{theorem}{Theorem}

\newtheorem{proposition}{Proposition}
\newtheorem{corollary}{Corollary}

\newcommand*{\C}{\mathbb{C}}
\newcommand*{\R}{\mathbb{R}}
\newcommand*{\Q}{\mathbb{Q}}
\newcommand*{\Z}{\mathbb{Z}}
\newcommand*{\N}{\mathbb{N}}

\newcommand{\comment}[1]{}
\title[On variants of Chebyshev's conjecture]%
     {On variants of Chebyshev's conjecture} 
\author[M. Suzuki]{Masatoshi Suzuki}
\subjclass[]{
11M26 
11N05 
}
\keywords{
Chebyshev's bias, 
Riemann Hypothesis, 
Generalized Riemann Hypothesis, 
screw functions
}
\AtBeginDocument{%
\begin{abstract}
We show that the sign constancy for the values of certain weighted summatory functions 
of the von Mangoldt function implies the Riemann hypothesis or 
the generalized Riemann hypothesis for Dirichlet $L$-functions. 
While such sign constancy is challenging to establish individually, 
we prove that the summatory functions under study 
have constant signs on average.
\end{abstract}
\maketitle
}
\begin{document}

\section{Introduction} 

\subsection{Results on the Riemann zeta function} 

Let $\pi(x)$ be the number of prime numbers $\leq x$. 
According to the prime number theorem, 
the leading term in the asymptotic behavior of $\pi(x)$ as $x \to \infty$ 
is the logarithmic integral ${\rm li}(x)=\int_{0}^{x}dt/\log t$. 
The difference $\pi(x) - {\rm li}(x)$ had long been conjectured to be negative for all $x>2$.  
This conjecture was disproven in 1914 by Littlewood, 
who showed that it changes sign infinitely often \cite{HL16}. 
However, Rubinstein and Sarnak \cite{LS94} proved 
that the logarithmic density of $x$ for which $\pi(x) - {\rm li}(x)$ is negative  
is close to one 
assuming the Riemann Hypothesis (RH) holds -- 
that is, all nontrivial zeros of $\zeta(s)$ have real part $1/2$ -- 
 and the set of positive imaginary parts of the nontrivial zeros of the Riemann zeta function $\zeta(s)$ 
is linearly independent over $\Q$. 
Furthermore, 
as mentioned in Pintz \cite{Pi91} and as proved by Johnston \cite{Jo23}, 
the RH is equivalent to the negativity 
\begin{equation} \label{EQ_101}
\int_2^x (\pi(t)-{\rm li}(t))\,dt <0 \quad (x>2).
\end{equation}
Johnston also shows that the RH is equivalent to the negativity 
\begin{equation} \label{EQ_102}
\int_2^x (\vartheta(t)-t)\,dt <0 \quad (x>2),
\end{equation}
where $\vartheta(x)=\sum_{p \leq x}\log p$, the first Chebyshev function. Here and 
in what follows, the letter $p$ will always denote a prime number. 
In contrast, no such bias is observed in the difference between 
the second Chebyshev function 
\[
\psi(x):=\sideset{}{'}\sum_{n \leq x}\Lambda(n) 
\]
and its leading term $x$ as $x \to \infty$, where $\Lambda(n)$ is the von Mangoldt function defined by 
$\Lambda(n)=\log p$ if $n=p^k$ with $k \in \Z_{>0}$ 
and $\Lambda(n)=0$ otherwise, and $\sideset{}{'}\sum_{n \leq x}a_n$ means 
$\sum_{n \leq x}a_n-(1/2)a_x$ when $x$ is a prime power 
and $\sum_{n \leq x}a_n$ otherwise. 

On the other hand, a negative bias is observed in the difference 
between the weighted Chebyshev function
\begin{equation} \label{EQ_103}
\psi_{1/2}(x):=\sideset{}{'}\sum_{n \leq x} \frac{\Lambda(n)}{\sqrt{n}}
\end{equation}
and its leading term $2\sqrt{x}$ as $x \to \infty$.  
Regarding this observation, using the explicit formula
\begin{equation} \label{EQ_104}
\aligned 
\psi_{1/2}(x) - 2\sqrt{x}
= 
- \sum_{\rho} \frac{x^{\rho-1/2}}{\rho-1/2} 
- \frac{\zeta'}{\zeta}\left(\frac{1}{2}\right)
+o(1)
\endaligned 
\end{equation}
in \eqref{EQ_206} below, 
together with Littlewood's result \cite[Theorem 5.6]{HL16}
\[
\sum_{\rho} \frac{x^{\rho}}{\rho}
= \Omega_{\pm}(\sqrt{x}\log\log\log x), 
\]
we can show that $\psi_{1/2}(x) - 2\sqrt{x}$ changes sign infinitely often 
as well as $\pi(x)-{\rm li}(x)$,  
where the sum $\sum_\rho$ runs over all nontrivial zeros of $\zeta(s)$ counted with multiplicity,  
and is defined as $\sum_\rho = \lim_{T \to \infty}\sum_{|\Im(\rho)| \leq T}$. 
The Landau symbol 
$f(x)=o(g(x))$ means that $|f(x)|/g(x)$ tends to zero as $x \to \infty$. 
Furthermore, 
$f(x)=\Omega_{+}(g(x))$ (resp. $f(x)=\Omega_{-}(g(x))$) means that $\limsup_{x \to \infty} f(x)/g(x)>0$ 
(resp.  $\liminf_{x \to \infty} f(x)/g(x)<0$), 
and $f(x)=\Omega_{\pm}(g(x))$ means that both $f(x)=\Omega_{+}(g(x))$ and $f(x)=\Omega_{-}(g(x))$ hold.  
By the second functional equation 
in \eqref{EQ_204}, 
if $\rho$ is a nontrivial zero of $\zeta(s)$, then so is $\bar{\rho}$. 
Therefore, both $\sum_\rho x^{\rho}/\rho$ and 
$\sum_\rho x^{\rho-1/2}/(\rho-1/2)$ are real valued. 

However, assuming the RH, 
$\psi_{1/2}(x) - 2\sqrt{x}$ is negative on average, 
because 
\[
\int_{0}^{x} (\psi_{1/2}(y) - 2\sqrt{y}\,) \frac{dy}{y}
 = - \frac{\zeta'}{\zeta}\left(\frac{1}{2}\right) \log x 
- \sum_{\rho} \frac{x^{\rho-1/2}}{(\rho-1/2)^2} +o(1) 
\]
holds, $(\zeta'/\zeta)(1/2)=2.68609\dots$ (see \eqref{EQ_301} and \eqref{EQ_302} below), 
and the sum on the right-hand side is bounded under the RH 
by $\sum_\rho |\rho|^{-2}<\infty$ (see \S2.2). 
This bias in the difference $\psi_{1/2}(x) - 2\sqrt{x}$ is also a sufficient condition for the RH 
as in the case of $\pi(x)-{\rm li}(x)$ and $\vartheta(x)-x$:

\begin{theorem} \label{thm_1}
The following three statements are equivalent:
\begin{enumerate}
\item The RH holds. 
\item There exists an $x_0 \geq 2$ such that
\begin{equation} \label{EQ_105}
\aligned 
\int_{0}^{x} & (\psi_{1/2}(y) - 2\sqrt{y}\,) \frac{dy}{y} \\
&=
\sum_{n \leq x} \frac{\Lambda(n)}{\sqrt{n}}\log\frac{x}{n}  - 4 \sqrt{x} ~\leq~  0
\endaligned 
\end{equation}
holds for all $x \geq x_0$. 
\item It holds that 
\begin{equation} \label{EQ_106}
\lim_{x \to \infty}
\left[
\sum_{n \leq x} \frac{\Lambda(n)}{\sqrt{n}}
\left(1-\frac{\log n}{\log x} \right)
- \frac{4\sqrt{x}}{\log x}
\right] = -\frac{\zeta'}{\zeta}\left(\frac{1}{2}\right). 
\end{equation}
\end{enumerate}
\end{theorem}

The prime on the sum is unnecessary on the left-hand side of \eqref{EQ_105}, 
since $\log(x/n)$ vanishes when $n=x$. 
Taking the average 
\[
\frac{1}{x}\int_{1}^{x} \frac{1}{2}\sqrt{\frac{x}{u}} \log \sqrt{\frac{x}{u}} \, du = 
1 - \frac{1}{2\sqrt{x}}(\log x+2)
\]
into account, the inequality in \eqref{EQ_105} 
can be rewritten as an inequality for the sum of $\Lambda(n)$ 
with a weight that is approximately one on average:
\[
\sum_{n \leq x}\Lambda(n)\cdot \frac{1}{2}\sqrt{\frac{x}{n}}\log\sqrt{\frac{x}{n}}  - x~\leq~  0. 
\]
This representation is sometimes used for its apparent simplicity as in Theorem \ref{thm_4} below. 

Considering \cite{LS94}, 
we are interested in the logarithmic density of $x$ 
for which $\psi_{1/2}(x)-2\sqrt{x} \leq 0$ holds, 
but we do not study that topic in this paper. 
\medskip

The equivalence between the first and second statements in Theorem~\ref{thm_1} 
is quite similar to that between the RH and either \eqref{EQ_101} or \eqref{EQ_102}. 
However, an analog of \eqref{EQ_106} for $\pi(x) - \mathrm{li}(x)$ 
or $\vartheta(x) - x$ does not appear to have been previously noted in the literature. 
Moreover, it should be mentioned that 
Theorem~\ref{thm_1} was not discovered as an analog of \eqref{EQ_101}; 
rather, it arose in connection with the screw functions discussed below.
\medskip


Theorem \ref{thm_1} can be interpreted 
in terms of the screw function $g_\zeta(t)$ for $\zeta(s)$ studied in \cite{Su23}.  
It is a function on $\R$ defined by 
\[
g_\zeta(t) := g_0(t) + g_\infty(t)
\]
with 
\begin{equation} \label{EQ_107}
\aligned 
g_0(t) 
&:= \sum_{n \leq \exp(|t|)} \frac{\Lambda(n)}{\sqrt{n}}(|t|-\log n)-4(e^{t/2} + e^{-t/2}-2) , \\
g_\infty(t) 
& := -
\frac{|t|}{2}\left( \frac{\Gamma'}{\Gamma}\left(\frac{1}{4}\right)-\log \pi \right) 
- \frac{1}{4}\Bigl(\Phi(1,2,1/4)-e^{-|t|/2}\Phi(e^{-2|t|},2,1/4)\Bigr), 
\endaligned 
\end{equation}
where $\Phi(z,s,a) = \sum_{n=0}^{\infty} z^n(n+a)^{-s}$ 
is the Hurwitz--Lerch zeta function for $|z|<1$ and $a \not=0, -1,-2,\cdots$. 
It is known that $-g_\zeta(t) \geq 0$ for all $t \in \R$ 
if and only if the RH holds \cite[Theorem 1.7]{Su23}. 
The nonnegativity comes from the fact that $g_\zeta(t)$ 
is a screw function in the sense of \cite[p. 189]{KrLa77} under the RH. 
From the proof of the integral formula
\[
\int_{0}^{\infty} g_\zeta(t) e^{izt} \, dt = \frac{1}{z^2} \frac{\xi'}{\xi}\left(\frac{1}{2}-iz\right)
\]
in \cite[Section 2.1]{Su23}, 
the function 
$g_0(t)$ corresponds to the non-archimedean factor $s(s-1)\zeta(s)$ 
and $g_\infty(t)$ corresponds to the archimedean factor $\pi^{-s/2}\Gamma(s/2)$ 
of the Riemann xi-function 
\begin{equation} \label{EQ_108}
\xi(s):=\frac{1}{2}\,s(s-1)\pi^{-s/2}\Gamma(s/2)\zeta(s).
\end{equation}
According to the Kre\u{\i}n--Langer criterion \cite[Satz 5.9]{KrLa77}, 
$g_0(t)$ cannot be a screw function even assuming the RH, 
because the logarithmic derivative of $s(s-1)\zeta(s)$ can never belong to the Nevanlinna class. 
However, if we note the relation
\begin{equation} \label{EQ_109} 
g_0(\log x) = \sum_{n \leq x} \frac{\Lambda(n)}{\sqrt{n}}\log\frac{x}{n} -
4\left(\sqrt{x}+\frac{1}{\sqrt{x}}-2\right), 
\end{equation}
then we find that the RH is equivalent to $g_0(t)$ having a constant sign. 

\begin{corollary} \label{cor_1}
Let $g_0(t)$ be the non-archimedean part of $g_\zeta(t)$ defined in \eqref{EQ_107}. 
Then, the RH holds if and only if there exists $t_0 > 0$ such that 
$-g_0(t) \geq 0$ holds for all $t \geq t_0$. 
\end{corollary}


Noting \eqref{EQ_104}, 
if we replace $4\sqrt{x}$ in \eqref{EQ_105} with 
$2 \sum_{n \leq x} \Lambda(n)/\sqrt{n}$ 
and express $2$ as $\log\exp(2)$, 
we obtain the following  as an analog of Theorem \ref{thm_1}. 

\begin{theorem} \label{thm_2}
Suppose that there exists an $x_0 \geq 2$ such that
\begin{equation} \label{EQ_110}
\sum_{n \leq xe^2} \frac{\Lambda(n)}{\sqrt{n}}\log\frac{x}{n} ~\leq~  0
\end{equation}
holds for all $x \geq x_0$. 
Then, the RH holds. 
Furthermore, it holds that 
\begin{equation} \label{EQ_111}
\lim_{x \to \infty}
\sum_{n \leq xe^2} \frac{\Lambda(n)}{\sqrt{n}}
\left(1- \frac{\log n}{\log x}\right) = -\frac{\zeta'}{\zeta}\left(\frac{1}{2}\right)
\end{equation}
if and only if the RH holds and 
the estimate 
\begin{equation} \label{EQ_112} 
 \sum_{\rho} \frac{x^{\rho}}{\rho} = o(\sqrt{x}\log x) 
\end{equation}
is valid as $x \to \infty$.  
\end{theorem}

The right-hand side of \eqref{EQ_111} is negative (see \eqref{EQ_302} below). 
Therefore, \eqref{EQ_111} implies that 
there exists $x_0 \geq 2$ such that \eqref{EQ_110} holds for all $x \geq x_0$, 
\medskip

For the sum on the left-hand side of \eqref{EQ_112}, 
the best known estimate under the RH is
\[
\sum_{\rho} \frac{x^{\rho}}{\rho}
\ll \sqrt{x}(\log x)^2 \quad \text{as $x \to \infty$}
\]
\cite[pp. 82--85]{In90}, 
but it is not enough to prove \eqref{EQ_110} for large $x$, 
where the Vinogradov symbol $f(x) \ll g(x)$ mean that there exists a positive constant $M$ such that 
$|f(x)| \leq M g(x)$ holds for a prescribed range of $x$. 
We use the the Landau symbol $f(x)=O(g(x))$ in the same meaning. 
In the above sense, \eqref{EQ_110} and \eqref{EQ_111}  may provide a condition stronger than the RH. 
For comparison, we note that Montgomery's conjecture 
\[
\limsup_{x \to \infty} \frac{\sum_{\rho} x^{\rho}/\rho}{\sqrt{x}(\log\log\log x )^2}
= \frac{1}{2\pi}
\quad \text{and} \quad 
\liminf_{x \to \infty} \frac{\sum_{\rho} x^{\rho}/\rho}{\sqrt{x}(\log\log\log x )^2}
= -\frac{1}{2\pi}
\]
\cite[p.16]{Mo80} implies the RH and the estimate \eqref{EQ_112}. 
The equivalence of \eqref{EQ_111} and \eqref{EQ_112} 
can be viewed as analogous to Akatsuka \cite[Theorem 1]{Aka17}, 
which in turn parallels Conrad \cite[Theorem 1.3]{Con05} for $\zeta(s)$. 
Taking the logarithm of \cite[Corollary 4.4]{Aka17} 
and using \cite[Lemma 2.1]{Aka17}, 
Akatsuka's result asserts that  
\[
\lim_{x \to \infty}
\left[ \sum_{n \leq x}  \frac{\Lambda(n)}{\sqrt{n}\log n}
-\lim_{\varepsilon \to 0+}\left(\int_{0}^{1-\varepsilon} + \int_{1+\varepsilon}^{\sqrt{x}}\right)
\frac{du}{\log u} \right] 
= \log\left(-\zeta\left(\frac{1}{2}\right)\right)
\]
is equivalent to the validity of the RH and \eqref{EQ_112}, 
where $-\zeta(1/2)=1.46035\cdots$. 
Recently, Akatsuka showed that the equation obtained 
by replacing the upper limit $\sqrt{x}$ of the integral on the left-hand side with 
$\sqrt{\vartheta(x)}$ is equivalent to the RH (\cite[Theorem 3]{Aka24}).
\medskip

A sequence $\{a_n\}$ is called Riesz summable of order $k$ to $a$ if 
\[
\lim_{x \to \infty} \sum_{n \leq x} a_n \left(1-\frac{\log n}{\log x}\right)^k=a 
\]
(cf. \cite[p.158]{MV07}). From \eqref{EQ_301} below, 
$\{\Lambda(n)/\sqrt{n}\}$ cannot be Riesz summable of order one even if the RH is assumed. 
However, Theorem \ref{thm_2} asserts that if we suppose a condition stronger than the RH, 
then a slightly modified Riesz sum of order one has a limit, 
and conversely, the existence of such a limit implies the RH. 
\medskip

If we rewrite \eqref{EQ_110} as  
\begin{equation} \label{EQ_113}
\sum_{n \leq x} \Lambda(n)\sqrt{\frac{x}{n}}\log\frac{x}{n}
\quad \leq~
\sum_{x < n \leq xe^2} \Lambda(n)\sqrt{\frac{x}{n}}\log\frac{n}{x}
\end{equation}
it could be interpreted as a bias depending on the size of the weighted prime powers. 

\subsection{Results related to Chebyshev's conjecture} 

For any pair of integers $q$ and $a$ with $(q,a)=1$, 
let $\pi(x; q, a)$ denote the number of primes $p \leq x$ such that $p \equiv a$ mod $q$. 
The leading term in the asymptotic behavior of $\pi(x; q, a)$
as $x \to \infty$ is ${\rm li}(x)/\varphi(q)$, 
where $\varphi(q)$ is the Euler totient function. 
As in the case of the difference $\pi(x)-{\rm li}(x)$, 
Chebyshev's conjecture that $\pi(x;4,1)-\pi(x;4, 3)$ is always nonpositive
was disproved in \cite{HL16}, 
but it has been shown that the logarithmic density of the set of $x > 2$ 
for which $\pi(x;4,1)-\pi(x; 4, 3)$ is negative is close to one \cite{LS94}. 

Now, considering the similarity between 
$\pi(x)-{\rm li}(x)$ and $\pi(x;4,1)-\pi(x;4, 3)$, 
we introduce the weighted function
\[
\psi_{1/2}(x;q,a):=\sideset{}{'}\sum_{{n \leq x}\atop{n \equiv a\,{\rm mod}\,q}} \frac{\Lambda(n)}{\sqrt{n}} 
\]
for integers $q$ and $a$ with $(q,a)=1$. 
This is a Chebyshev function version of the weighted counting function 
$\pi_{1/2}(x;q,a)$ studied in Aoki and Koyama \cite{AK22}. 
Then, a negative bias is observed in the difference 
\[
\psi_{1/2}(x;4,1)-\psi_{1/2}(x;4,3) 
= \sideset{}{'}\sum_{n \leq x} \frac{\Lambda(n)\chi_4(n)}{\sqrt{n}}, 
\]
where $\chi_4$ is the non-principal Dirichlet character modulo four. 
Furthermore, we find that $\psi_{1/2}(x;4,1)-\psi_{1/2}(x;4,3) $ is negative on average 
if and only if all zeros of the Dirichlet $L$-function $L(s,\chi_4)$ 
in the critical strip $0 \leq \Re(s) \leq 1$ lie on the line $\Re(s)=1/2$. 
The latter is nothing but the Generalized Riemann Hypothesis (GRH) for $L(s,\chi_4)$. 

\begin{theorem} \label{thm_3}
The following three statements are equivalent:
\begin{enumerate}
\item The GRH for $L(s,\chi_4)$ holds. 
\item There exists an $x_0 \geq 2$ such that
\begin{equation} \label{EQ_114}
\aligned 
\int_{0}^{x} 
(\psi_{1/2}(y;4,1)&-\psi_{1/2}(y;4,3)) \frac{dy}{y} \\
& =
\sum_{n \leq x} \frac{\Lambda(n)\chi_4(n)}{\sqrt{n}}\log\frac{x}{n} ~\leq~  0
\endaligned 
\end{equation}
holds for all $x \geq x_0$. 
\item It holds that 
\begin{equation} \label{EQ_115}
\lim_{x \to \infty}
\sum_{n \leq x} \frac{\Lambda(n)\chi_4(n)}{\sqrt{n}}
\left(1- \frac{\log n}{\log x}\right) = -\frac{L'}{L}\left(\frac{1}{2},\chi_4\right). 
\end{equation}
\end{enumerate}
\end{theorem}

Unlike \cite[Theorem 1.3]{Con05}, \cite[Theorem 1]{Aka17}, 
and Theorem \ref{thm_2},  
which include claims stronger than the GRH, 
it should be noted that \eqref{EQ_115} is equivalent to the GRH 
for $L(s, \chi_4)$. 
On the other hand, similar to Theorem \ref{thm_1}, 
the nonpositivity in \eqref{EQ_114} can be interpreted as a property of 
the non-archimedean part of the screw function 
associated with $L(s,\chi_4)$ studied in \cite{Su22}, 
but we omit the details, as they may be easily inferred. 
\medskip

Based on his observation of the negative bias in 
the difference $\pi(x;4,1)-\pi(x;4,3)$, 
Chebyshev conjectured 
\begin{equation} \label{EQ_116}
\lim_{x\to\infty} \sum_{p>2} (-1)^{(p-1)/2} \cdot e^{-p/x}=-\infty.
\end{equation}
The sum on the left-hand side of \eqref{EQ_116} can be written as 
\[
\sum_{p \equiv 1\,{\rm mod}\, 4} e^{-p/x} \quad 
 -  
\sum_{p \equiv 3\,{\rm mod}\, 4} e^{-p/x}. 
\]
Here $e^{-p/x} \to 1$ as $x \to \infty$ termwise. 
Therefore, Chebyshev's conjecture \eqref{EQ_116} 
suggests that there are ``more'' primes of the form 
$p=4n+3$ than those of the form $p=4n+1$. 
Subsequently, Hardy and Littlewood~\cite{HL16} and Landau~\cite{La18} 
proved that \eqref{EQ_116} and 
\begin{equation} \label{EQ_117}
\lim_{x\to\infty} \sum_{p>2} (-1)^{(p-1)/2} \log p \cdot e^{-p/x}=-\infty
\end{equation}
are equivalent to 
the GRH for $L(s, \chi_{4})$. 
Knapowski and Tur\'{a}n \cite{KnTu69} showed that replacing the weight function 
$\exp(-p/x)$ in \eqref{EQ_116} and \eqref{EQ_117} 
with $\exp(-(\log(p/x))^2)$ 
is likewise equivalent to the GRH for $L(s, \chi_{4})$. 
Fujii~\cite{Fu88} proved that when $\exp(-p/x)$ is replaced with $\exp(-(p/x)^\alpha)$ 
(resp. $2K_0(2\sqrt{p/x})$), 
the result in \eqref{EQ_116} (resp. \eqref{EQ_117}) is equivalent to the GRH for $L(s, \chi_{4})$ 
if $0<\alpha<4.19$, and he conjectured that this would hold for general $\alpha>0$ as well.
 Platt and Trudgian~\cite{PT18} extend Fujii's result for \eqref{EQ_116} to $0<\alpha<20.4$. 
In addition, 
a few other weight functions were considered in 
the works of Hardy--Littlewood \cite{HL16} and Fujii~\cite{Fu88}.
Based on the above developments and Theorem \ref{thm_3}, 
considering the weight function $\sqrt{x/p}\log(x/p)\mathbf{1}_{p \leq x}$ 
in \eqref{EQ_117} leads to the following result. 

\begin{theorem} \label{thm_4}
The GRH for $L(s,\chi_{4})$ holds if and only if
\begin{equation} \label{EQ_118}
\lim_{x \to \infty} 
\sum_{2<p \leq x} (-1)^{(p-1)/2} \log p \cdot\sqrt{\frac{x}{p}}\,\log\frac{x}{p} 
= -\infty.
\end{equation}
\end{theorem}

As a result comparable to Theorem \ref{thm_4}, 
we note that the asymptotic formula 
\begin{equation} \label{EQ_119}
\sum_{2<p \leq x} (-1)^{(p-1)/2}  \sqrt{\frac{x}{p}}~ \sim ~ \frac{1}{2}\sqrt{x}\log \log x
\quad \text{as $x \to \infty$}
\end{equation}
holds 
if and only if the GRH for $L(s,\chi_4)$ holds and 
\eqref{EQ_112} is satisfied for the nontrivial zeros of $L(s,\chi_4)$ 
by combining \cite[Theorem 1.1, (1.3)]{AK22} and \cite[Theorem 6.2]{Con05}. 
According to the proof of Theorem \ref{thm_4}, 
the GRH for $L(s,\chi_{4})$ is equivalent to the left-hand side 
of \eqref{EQ_118} being asymptotically equal to $-(\sqrt{x}/4)(\log x)^2$; 
thus, unlike \eqref{EQ_119}, 
it does not yield a result stronger than the GRH. 
However, an analog of Theorem \ref{thm_4} for $\zeta(s)$ 
yields a conclusion stronger than the RH, similar to Theorem \ref{thm_2}.

\begin{theorem} \label{thm_5}
The RH holds assuming 
\begin{equation} \label{EQ_120}
\lim_{x \to \infty} 
\sum_{p \leq xe^2} \log p \cdot \sqrt{\frac{x}{p}}\,\log\frac{x}{p} 
= -\infty. 
\end{equation}
Conversely, if the RH holds and \eqref{EQ_112} is valid as $x \to \infty$, 
then \eqref{EQ_120} holds. 
\end{theorem}

\subsection{Results on averages over Dirichlet characters} 

As naturally expected, Theorem \ref{thm_3} generalizes to a result 
for general Dirichlet $L$-functions, as stated in Theorem \ref{thm_8} below.  
This generalization is also described by inequalities for oscillating sums.  
By considering the averages of these sums, we obtain the following result.  
As in~\eqref{EQ_113}, this average can be rewritten as the difference of two sums, 
each having a constant sign in the first and second halves.  
For the following results, for an imprimitive Dirichlet character $\chi$, 
we define the GRH to hold for $L(s,\chi)$ as the GRH to hold for $L(s,\chi^\ast)$, 
where $\chi^\ast$ is a primitive Dirichlet character inducing $\chi$. 

\begin{theorem} \label{thm_6}
Let $q$ be a positive integer, 
and let $\beta=\beta(q)$ be a real number such that 
$1/2 \leq \beta <1$ and 
$L(\sigma,\chi)\neq 0$ for all $\sigma >\beta$ 
and all Dirichlet character $\chi$ modulo $q$.  
\begin{enumerate} 
\item Suppose that there exists $x_0 \geq 2$ such that either
\begin{equation} \label{eq_0613_2}
\sum_{{n \leq x}\atop{n \equiv 1 \,{\rm mod}\, q}} 
\frac{\Lambda(n)}{\sqrt{n}} \log\frac{x}{n}
- \frac{4\sqrt{x}}{\varphi(q)}  
\end{equation}
has a constant sign for all $x \geq x_0$, or 
\begin{equation} \label{EQ_122}
\sum_{{n \leq xe^2}\atop{n \equiv 1 \,{\rm mod}\, q}} \frac{\Lambda(n)}{\sqrt{n}}\log\frac{x}{n} 
\end{equation}
has a constant sign for all $x \geq x_0$. 
Then $L(s,\chi) \not=0$ in the right half-plane $\Re(s)>\beta$ for all $\chi$ modulo $q$. 
In particular, if $\beta=1/2$, the GRH for $L(s,\chi)$ holds for all 
Dirichlet characters $\chi$ modulo $q$. 
\item 
Suppose that $L(1/2,\chi)\not=0$ for all Dirichlet character $\chi$ modulo $q$.   
Then, we have
\begin{equation} \label{EQ_123}
\aligned 
\lim_{x \to \infty} &
\left(
\sum_{{n \leq x}\atop{n \equiv 1 \,{\rm mod}\, q}} 
\frac{\Lambda(n)}{\sqrt{n}}
\left(1-\frac{\log n}{\log x} \right)
- \frac{4\sqrt{x}}{\varphi(q)\log x}
\right) \\ 
& \qquad \qquad \qquad \qquad \qquad \quad 
=  -\frac{1}{\varphi(q)} \sum_{\chi\,{\rm mod}\, q} \frac{L'}{L}\left(\frac{1}{2},\chi\right)
\endaligned 
\end{equation}
if and only if the GRH  for $L(s,\chi)$ holds for all Dirichlet characters $\chi$ modulo $q$.  
\item Suppose that $L(1/2,\chi)\not=0$ for all Dirichlet character $\chi$ modulo $q$.   
Then, we have
\begin{equation} \label{EQ_124} 
\lim_{x \to \infty}
\sum_{{n \leq xe^2}\atop{n \equiv 1 \,{\rm mod}\, q}}  \frac{\Lambda(n)}{\sqrt{n}}
\left(1- \frac{\log n}{\log x}\right) = -\frac{1}{\varphi(q)} 
\sum_{\chi\,{\rm mod}\,q}\frac{L'}{L}\left(\frac{1}{2},\chi\right)
\end{equation}
if and only if the GRH for $L(s, \chi)$ holds for all Dirichlet characters $\chi$ modulo $q$   
and
\begin{equation} \label{EQ_125} 
\sum_{\chi\,{\rm mod}\, q} \sum_{\rho_{\chi^\ast}} \frac{x^{\rho_{\chi^\ast}}}{\rho_{\chi^\ast}} 
= o(\sqrt{x} \log x)
\end{equation}
holds as $x \to \infty$ for nontrivial zeros of $L(s, \chi^\ast)$ 
for all $\chi$ modulo $q$, 
where $\chi^\ast$ is a primitive Dirichlet character that induces \( \chi \).
As well as the sum $\sum_\rho$ above, 
the sum $\sum_{\rho_{\chi^\ast}}$ runs over all nontrivial zeros of $L(s,\chi^\ast)$ counted with multiplicity, 
and is defined as $\sum_{\rho_{\chi^\ast}} = \lim_{T \to \infty}\sum_{|\Im(\rho_{\chi^\ast})| \leq T}$.  
\item Let $m_\chi$ denote the order of the zero of $L(s,\chi)$ at $s = 1/2$. 
Then, we have
\begin{equation} \label{eq_0612_6}
\aligned 
\lim_{x \to \infty} & 
\left(\frac{1}{\log x}
\sum_{{n \leq x}\atop{n \equiv 1 \,{\rm mod}\, q}} 
\frac{\Lambda(n)}{\sqrt{n}}
\left(1-\frac{\log n}{\log x} \right)
- \frac{4\sqrt{x}}{\varphi(q)(\log x)^2}
\right) \\ 
& \qquad \qquad \qquad \qquad \qquad \quad 
=  -\frac{1}{2}\sum_{\chi\,{\rm mod}\,q} m_\chi
\endaligned 
\end{equation}
if and only if the GRH  for $L(s,\chi)$ holds for all Dirichlet characters $\chi$ modulo $q$.  
\item  Under the notation of the third and fourth statements, we have
\begin{equation} \label{eq_0612_7}
\lim_{x \to \infty} \,  \frac{1}{\log x}
\sum_{{n \leq xe^2}\atop{n \equiv 1 \,{\rm mod}\, q}}  \frac{\Lambda(n)}{\sqrt{n}}
\left(1- \frac{\log n}{\log x}\right) = -\frac{1}{2}\sum_{\chi\,{\rm mod}\,q} m_\chi
\end{equation}
if and only if the GRH for $L(s, \chi)$ holds for all Dirichlet characters $\chi$ modulo $q$   
and 
\begin{equation} \label{eq_0613_3} 
\sum_{\chi\,{\rm mod}\, q} \sum_{\rho_{\chi^\ast}} \frac{x^{\rho_{\chi^\ast}}}{\rho_{\chi^\ast}} 
= o(\sqrt{x} (\log x)^2)
\end{equation}
holds as $x \to \infty$ for nontrivial zeros of $L(s, \chi^\ast)$ 
for all $\chi$ modulo $q$. 
\end{enumerate}
\end{theorem}

If $L(1/2,\chi)\not=0$ for all Dirichlet character $\chi$ modulo $q$, then the right-hand sides of \eqref{eq_0612_6} and \eqref{eq_0612_7} vanish. 
Hence, the second and fourth claims of Theorem~\ref{thm_6},  
as well as the third and fifth claims, are consistent. 
\medskip

As for \eqref{EQ_110}, nonpositivity of \eqref{EQ_122} 
can be expressed as an inequality for nonnegative sums 
\[
\sum_{{n \leq x}\atop{n \equiv 1 \,{\rm mod}\, q}} 
\Lambda(n)\sqrt{\frac{x}{n}}\log\frac{x}{n}
\quad \leq~
\sum_{{x<n \leq xe^2}\atop{n \equiv 1 \,{\rm mod}\, q}} 
\Lambda(n)\sqrt{\frac{x}{n}}\log\frac{n}{x}. 
\]
The previous \eqref{EQ_113} corresponds to the case $q=1$ of this inequality.

According to the second and third statements of Theorem~6, the sign of the sums \eqref{eq_0613_2} and \eqref{EQ_122} for sufficiently large \( x \), which are assumed in the first statement, 
should be determined by the constant
\[
C(q) := \sum_{\chi\,{\rm mod}\, q} \frac{L'}{L}\left(\frac{1}{2},\chi\right) \in \mathbb{R}.
\]
For simplicity, let $q$ be an odd prime. 
Then, taking into account that there are $(q-1)/2-1$ even primitive Dirichlet characters 
and $(q-1)/2$ odd primitive Dirichlet characters, we obtain the formula
\[
C(q) 
= \frac{q-1}{2}(\log 8\pi + C_0) - \frac{q-2}{2}\log q
+ \frac{\log q}{\sqrt{q}-1}
\]
from \eqref{EQ_302}, \eqref{EQ_406}, and \eqref{eq_0613_1} below, 
where $C_0$ is the Euler--Mascheroni constant.  
It follows that $C(q) < 0$ for all odd primes $\leq 47$, 
while $C(q) > 0$ for all primes $\geq 53$.
 
The sign constancy of \eqref{eq_0613_2} and \eqref{EQ_122} is conditional. 
However, we find that they are unconditionally negative on average for all sufficiently large $x>1$.

\begin{theorem} \label{thm_7} We have 
\begin{equation} \label{EQ_126}
\aligned 
\sum_{3 \leq q \leq Q} &  \varphi(q) \left(
\sum_{{n \leq x}\atop{n \equiv 1 \,{\rm mod}\, q}}  \frac{\Lambda(n)}{\sqrt{n}}\log\frac{x}{n} 
- \frac{4 \sqrt{x}}{\varphi(q)}\right) \\
& \qquad \qquad \qquad \qquad = 4\sqrt{x}\, \left[ \frac{1}{9}x\,(1+o(1)) -  Q \right] \quad \text{as $x \to \infty$}
\endaligned 
\end{equation}
and
\begin{equation} \label{EQ_127}
\sum_{3 \leq q \leq Q} \varphi(q)
\sum_{{n \leq xe^2}\atop{n \equiv 1 \,{\rm mod}\, q}} 
\frac{\Lambda(n)}{\sqrt{n}}\log\frac{x}{n} 
= - \frac{8}{9}\,e^{3}\,x\sqrt{x}\,(1+o(1)) \quad \text{as $x \to \infty$}
\end{equation}
if $Q \geq x$. 
In particular, the left-hand sides are negative for all sufficiently large $x>1$. 
\end{theorem}

The reason $q=1,2$ are excluded in both \eqref{EQ_126} and \eqref{EQ_127} is that, 
in those cases, the Dirichlet characters modulo $q$ are only the principal character.
While it is certainly desirable to obtain the negativity 
of the left-hand side of \eqref{EQ_126} or \eqref{EQ_127} 
for a large constant $Q \geq 3$, 
we are also interested in what nontrivial conclusions can be derived 
when the negativity of the left-hand side of \eqref{EQ_126} or \eqref{EQ_127} 
is obtained under a condition weaker than $Q \geq x$ 
with respect to $Q$. This remains a subject for future study. 
\medskip

This paper is organized as follows. 
In Section \ref{section_2}, we review the basic properties of 
the Riemann zeta function and Dirichlet $L$-functions. 
In Section \ref{section_3}, we prove 
the results for the Riemann zeta function:   
Theorems \ref{thm_1} and \ref{thm_2} 
and Corollary \ref{cor_1}. 
The key ingredients for the proof of Theorem \ref{thm_1} are 
an integral representation for a function involving 
the logarithmic derivative of the Riemann zeta function and 
a well-known result concerning the Laplace transform of a nonnegative valued function. 
Except for technical details, all results in this paper are shown 
using roughly the same approach as Theorem \ref{thm_1}.
In Section \ref{section_4}, we prove 
the results for Dirichlet $L$-functions: 
Theorems \ref{thm_3} and \ref{thm_4}. 
Although Theorem \ref{thm_5} is a result for the Riemann zeta function, 
it follows a similar way to that of Theorem \ref{thm_4}, 
so we include its proof in this section.
Finally, in Section \ref{section_5}, we prove Theorems \ref{thm_6} and \ref{thm_7} 
by combining the discussions in Sections \ref{section_3} and \ref{section_4}.
\medskip

\noindent
{\bf Acknowledgments}~
The author appreciates the referee's valuable comments, 
which led to the corrections of Theorems 6, 8, and 9,  
and contributed to improving the presentation of the paper. 
This work was supported by JSPS KAKENHI  Grant Number JP23K03050. 

\section{Preliminaries} \label{section_2}

In this section, we review a few elementary facts in number theory and 
the basic properties of the Riemann zeta function 
and Dirichlet $L$-functions referencing Titchmarsh \cite{Tit86} 
and Montgomery and Vaughan \cite{MV07}. 
In addition, we review a property of the Mellin transforms of nonnegative functions.

\subsection{Facts on prime numbers} 
We use Mertens' theorem \cite[Theorem 2.7]{MV07} in the form
\begin{equation} \label{EQ_201}
\sum_{p \leq x} \frac{\log p}{p} 
= \log x + O(1) \quad \text{as $x \to \infty$},  
\end{equation}
and the classical prime number theorem \cite[Theorem 6.9]{MV07} 
\begin{equation} \label{EQ_202}
\psi(x)-x= \sum_{\rho} \frac{x^{\rho}}{\rho} 
= O(x\exp(-c\sqrt{\log x}))\quad \text{as $x \to \infty$}, 
\end{equation}
where $\sum_\rho = \lim_{T \to \infty}\sum_{|\Im(\rho)| \leq T}$ as before. 

\subsection{Riemann zeta function} 
The Riemann zeta function $\zeta(s)$ 
is defined by the absolutely convergent series 
$\sum_{n=1}^{\infty}n^{-s}$ if $\Re(s)>1$, 
and extends to a meromorphic function on $\C$, 
which is analytic except for a simple pole at $s=1$ with residue one. 
Due to the Euler product formula 
$\zeta(s)=\prod_p (1-p^{-s})^{-1}$ \cite[(1.1.2)]{Tit86}, 
the logarithmic derivative of $\zeta(s)$ 
has the Dirichlet series 
\begin{equation} \label{EQ_203}
-\frac{\zeta'}{\zeta}(s) = \sum_{n=1}^{\infty} \frac{\Lambda(n)}{n^s}
\end{equation}
for $\Re(s)>1$ \cite[(1.1.8)]{Tit86}. 
The Riemann xi-function \eqref{EQ_108} is an entire function satisfying the functional equations 
\begin{equation} \label{EQ_204}
\xi(s)=\xi(1-s) \quad \text{and} \quad \xi(s)=\overline{\xi(\bar{s})}
\end{equation}
by \cite[(2.6.4)]{Tit86} and $\zeta(\sigma) \in \R$ for $\sigma>1$. 
Zeros of $\zeta(s)$ in the critical strip $0 \leq \Re(s) \leq 1$ are called nontrivial zeros 
and coincide with zeros of $\xi(s)$. 
Because $\xi(s)$ is an entire function of order one \cite[Theorem 2.12]{Tit86}, 
\begin{equation} \label{EQ_205}
\sum_\rho  |\rho|^{-1-\delta} < \infty \quad \text{if $\delta>0$},
\end{equation}
where $\rho$ runs over all nontrivial zeros of $\zeta(s)$ counted with multiplicity. 
We often use the explicit formula 
\begin{equation}\label{EQ_206}
\aligned 
\sideset{}{'}\sum_{n \leq x}\frac{\Lambda(n)}{n^s}
& = \frac{x^{1-s}}{1-s}
- \frac{\zeta^\prime}{\zeta}(s)
-  \sum_{\rho} \frac{x^{\rho-s}}{\rho-s}
+ \sum_{k=1}^\infty \frac{x^{-2k-s}}{2k+s}
\endaligned 
\end{equation}
for $x>1$, $s\not=1, \rho, -2k$ ($k \in \Z_{>0}$) 
\cite[Exercises 12.1.1, 4]{MV07}. 

\subsection{Dirichlet $L$-function} 

For a Dirichlet character $\chi$ modulo $q>1$, 
the Dirichlet $L$-function $L(s, \chi)$ 
is defined by the absolutely convergent series 
$\sum_{n=1}^{\infty}\chi(n)n^{-s}$ if $\Re(s)>1$. 
If $\chi$ is a principal character $\chi_0$, 
then $L(s, \chi_0)$ extends to a meromorphic function on $\C$, 
which is analytic except for a simple pole at $s=1$. 
If $\chi$ is not a principal character, 
then $L(s, \chi)$ extends to an entire function and
\begin{equation} \label{EQ_207}
L(1,\chi) \not=0
\end{equation}
\cite[Theorem 4.9]{MV07}. 
Due to the Euler product formula 
\begin{equation} \label{EQ_208}
L(s,\chi) = \prod_p (1-\chi(p)p^{-s})^{-1} \quad 
\end{equation}
for $\Re(s)>1$ \cite[(4.21)]{MV07},   
the logarithmic derivative of $L(s,\chi)$ 
has the absolutely convergent Dirichlet series 
\begin{equation} \label{EQ_209}
-\frac{L'}{L}(s,\chi) = \sum_{n=1}^{\infty} \frac{\Lambda(n)\chi(n)}{n^s}
\end{equation}
for $\Re(s)>1$ \cite[(4.25)]{MV07}. 

For a primitive Dirichlet character $\chi$ modulo $q>1$, 
we set
\[
\kappa
:=\kappa(\chi)
=
\begin{cases}
~0 & \text{if $\chi(-1)=1$}, \\
~1 & \text{if $\chi(-1)=-1$},
\end{cases}
\]
and 
\begin{equation} \label{EQ_210}
\xi(s,\chi) := \left(\frac{q}{\pi}\right)^{(s+\kappa)/2}\Gamma\left(\frac{s+\kappa}{2}\right) L(s,\chi).
\end{equation}
Then, it is an entire function satisfying the functional equation 
\begin{equation} \label{EQ_211}
\xi(s,\chi)=\varepsilon(\chi)\xi(1-s,\overline{\chi}), \quad 
\varepsilon(\chi) := \frac{\tau(\chi)}{i^\kappa\sqrt{q}}
\end{equation}
\cite[(10.17) and Corollary 10.8]{MV07}, 
where $\tau(\chi)$ is the Gauss sum of $\chi$. 
Zeros of $L(s,\chi)$ in the critical strip $0 \leq \Re(s) \leq 1$ 
and $s\not=0$ are called nontrivial zeros 
and coincide with the zeros of $\xi(s,\chi)$. 
The Generalized Riemann Hypothesis (GRH) 
claims that all nontrivial zeros lie on the critical line $\Re(s)=1/2$. 
Because $\xi(s,\chi)$ is an entire function of order one by \cite[(10.31)]{MV07}, 
\begin{equation} \label{EQ_212}
\sum_{\rho_\chi}  |\rho_\chi|^{-1-\delta} < \infty \quad \text{if $\delta>0$},
\end{equation}
where $\rho_\chi$ runs over all nontrivial zeros of $L(s,\chi)$ counted with multiplicity.
By an argument similar to the proof of \eqref{EQ_206}, we obtain 
\begin{equation}\label{EQ_213}
\sideset{}{'}\sum_{n \leq x}\frac{\Lambda(n)\chi(n)}{n^s}
=
- \frac{L^\prime}{L}(s,\chi) 
- \sum_{\rho_\chi} \frac{x^{\rho_\chi-s}}{\rho_\chi-s} 
+ \sum_{k=0}^{\infty} \frac{x^{-2k-\kappa-s}}{2k+\kappa+s}
\end{equation}
for $x>1$, $s\not=\rho_\chi, -2k-\kappa$ ($k \in \Z_{\geq 0}$) \cite[(6)]{Y91}. 

Suppose that $L(s, \chi)$ has a zero of order $m \geq 1$ at $s = 1/2$.  
By isolating the term of $\rho_\chi = 1/2$ in the second sum on the right-hand side of \eqref{EQ_213}, 
then taking the limit $s \to 1/2$, we obtain 
\begin{equation}\label{eq_0612_2}
\aligned 
\sideset{}{'}\sum_{n \leq x}\frac{\Lambda(n)\chi(n)}{\sqrt{n}}
& =
- m \log x
-\lim_{s \to 1/2} \left[\frac{L^\prime}{L}(s,\chi) - \frac{m}{s-1/2}\right] \\
& \quad - \sum_{\rho_\chi \not=1/2} \frac{x^{\rho_\chi-1/2}}{\rho_\chi-1/2}
+ \sum_{k=0}^{\infty} \frac{x^{-2k-\kappa-1/2}}{2k+\kappa+1/2}.
\endaligned 
\end{equation}

If $\chi$ is imprimitive, it is induced from a
primitive Dirichlet character $\chi^\ast$ modulo $q^\ast$, 
for some $q^\ast\mid q$, 
that is, 
\[
\chi(n) 
= 
\begin{cases}
~\chi^\ast(n) & (n,q)=1, \\
~0 & \text{otherwise}.
\end{cases}
\]
Therefore, 
\begin{equation} \label{EQ_214}
L(s,\chi) = L(s,\chi^\ast) \prod_{p \mid q}(1 - \chi^\ast(p) p^{-s})
\end{equation}
\cite[(9.1), (9.2)]{MV07}. For a principal character $\chi_0$ modulo $q$,   
$L(s,\chi_0) = \zeta(s) \prod_{p \mid q}(1 - p^{-s})$. 
We define the GRH to hold for $L(s,\chi)$ as the GRH to hold for $L(s,\chi^\ast)$. 

\subsection{Mellin transforms of nonnegative functions} 

We state a well-known property of the Laplace transform 
of nonnegative functions in terms of the Mellin transform 
\cite[Theorem 5b in Chap. II]{Wi41} 
(or \cite[Lemma 15.1]{MV07}). 
This is an analog of Landau's Theorem 
concerning Dirichlet series with nonnegative coefficients.

\begin{proposition} \label{prop_1} If $f(x)$ is nonnegative, 
then the real point of the axis of convergence of 
\[
F(s)=\int_{1}^{\infty} f(x) \, x^{-s} \, \frac{dx}{x} 
\]
is a singularity of $F(s)$. 
\end{proposition}

\section{Proof of Theorems \ref{thm_1} and \ref{thm_2}} \label{section_3}

\subsection{Proof of Theorem \ref{thm_1}} 
The equality in \eqref{EQ_105} is obtained by integrating \eqref{EQ_103}. 
First, we prove \eqref{EQ_105} and \eqref{EQ_106} assuming the RH. 
We  set 
\[
f(x) := \sum_{n \leq x} \frac{\Lambda(n)}{\sqrt{n}} \log\frac{x}{n}. 
\]
By applying \cite[Lemma 1]{So09} to $s=1/2$, we obtain 
\begin{equation} \label{EQ_301}
\aligned 
f(x) - 4 \sqrt{x}
& =
-\frac{\zeta'}{\zeta}\left(\frac{1}{2}\right) \log x 
- \left(\frac{\zeta'}{\zeta}\right)^\prime\left(\frac{1}{2}\right) \\
& \quad 
- \sum_\rho \frac{x^{\rho-1/2}}{(\rho-1/2)^2} 
- \sum_{k=1}^\infty \frac{x^{-2k-1/2}}{(2k+1/2)^2} 
\endaligned 
\end{equation}
for $x \geq 2$. 
The second sum on the second line of the right-hand side is clearly bounded  for $x \geq 2$. 
The first sum on the second line of the right-hand side is also bounded for $x \geq 2$, 
because $\rho-1/2$ is purely imaginary by the RH 
and \eqref{EQ_205}. On the other hand, 
\[
\frac{\xi'}{\xi}\left(\frac{1}{2}\right)=
\frac{\zeta'}{\zeta}\left(\frac{1}{2}\right)
+ \frac{1}{2}\left[ \frac{\Gamma'}{\Gamma}\left(\frac{1}{4}\right) - \log \pi \right] = 0
\]
by taking the logarithmic derivative of 
the first equation in \eqref{EQ_204}
and using $(\Gamma'/\Gamma)(1/4)=-\pi/2- 3\log 2 - C_0=-4.22745\dots<0$ 
(\cite[p. 411]{MV07}).
Therefore, 
\begin{equation} \label{EQ_302}
\frac{\zeta'}{\zeta}\left(\frac{1}{2}\right)=
\frac{1}{2}\left(\frac{\pi}{2}+\log 8\pi + C_0\right)=2.68609\dots>0.
\end{equation}
Hence, \eqref{EQ_105} holds for all sufficiently large $x \geq 2$. 
Furthermore, dividing both sides of \eqref{EQ_301} by $\log x$ 
and taking the limit as $x \to \infty$, 
we obtain \eqref{EQ_106}. 

By \eqref{EQ_302}, equality \eqref{EQ_106} implies that 
\eqref{EQ_105} holds for all sufficiently large $x$. 
Therefore, it remains to show that the RH holds 
if there exist $x_0 \geq 2$ such that \eqref{EQ_105} is valid for all $x \geq x_0$. 
The key point is the integral formula
\begin{equation} \label{EQ_303} 
\int_{1}^{\infty}  
(-f(x)) \,x^{-s+1/2} \, \frac{dx}{x} 
= 
\frac{1}{(s-1/2)^2} \frac{\zeta'}{\zeta}(s) 
\quad \text{for}~\Re(s)>1, 
\end{equation} 
which is derived as follows. 
For an interval $I \subset \R$, 
we denote by $\mathbf{1}_I$ the indicator function of $I$. 
Using the integral formula 
\begin{equation*} 
\int_{1}^{\infty} \frac{1}{\sqrt{n}}\log\frac{x}{n} \cdot \mathbf{1}_{[n,\infty)}(x)\cdot x^{-s+1/2} \, \frac{dx}{x} 
= \frac{1}{(s-1/2)^2} \, n^{-s} \quad \text{for}~\Re(s)>1/2  
\end{equation*}
obtained by a straightforward calculation for any $n \in \Z_{>0}$, 
we have
\begin{equation*} 
\aligned 
\int_{1}^{\infty}  (-f(x)) \, x^{-s+1/2} \, \frac{dx}{x} 
& = - \sum_{n=2}^{\infty}  \Lambda(n)
\int_{1}^{\infty} 
\frac{1}{\sqrt{n}}\log\frac{x}{n}\cdot \mathbf{1}_{[n,\infty)}(x)\cdot x^{-s+1/2} \, \frac{dx}{x}  \\
& = 
\frac{1}{(s-1/2)^2}\left(-\sum_{n=2}^{\infty} \Lambda(n)\, n^{-s} \right)
\endaligned 
\end{equation*} 
if $\Re(s)>1$,  
where the interchange of the order of integration and summation 
is justified by the absolute convergence of the series for $\Re(s)>1$ and Fubini's theorem. 
Therefore, we obtain \eqref{EQ_303} by \eqref{EQ_203}. 

On the other hand, 
$(\zeta'/\zeta)(s)$ has no poles in the half-line $(1/2,\infty)$ on the real axis 
except for the simple pole at $s=1$ with residue $-1$, 
because $\zeta(s)$ has the simple pole at $s=1$ with residue one and 
has no zeros in the half-line $(1/2,\infty)$ 
by the series expansion 
$(1-2^{1-s})\zeta(s) = \sum_{n=1}^{\infty} (-1)^{n-1}n^{-s}>0$ for $s>0$.  

By the integral formula \eqref{EQ_303} and $\int_{1}^{\infty}\sqrt{x} \cdot x^{-s+1/2} dx/x
= 1/(s-1)$ for $\Re(s)>1$, 
\begin{equation} \label{EQ_304}
\aligned  
\int_{1}^{\infty}  & 
(-f(x)+4\sqrt{x}\,) x^{-s+1/2} \, \frac{dx}{x} \\
& = 
\frac{1}{(s-1/2)^2} \frac{\zeta'}{\zeta}(s) +\frac{4}{s-1}
\quad \text{for}~\Re(s)>1.
\endaligned 
\end{equation} 
The simple pole of 
$(s-1/2)^{-2}(\zeta'/\zeta)(s)$ at $s=1$ 
is canceled out by the second term of the right-hand side. 
If we divide the left-hand side of \eqref{EQ_304} into 
\[
\left( \int_{x_0}^{\infty} + \int_{1}^{x_0} \right)
(-f(x)+4\sqrt{x}\,) x^{-s+1/2} \, \frac{dx}{x}, 
\]
the second integral defines an entire function. 
Applying Proposition~\ref{prop_1} to the first integral, 
we find that the abscissa of convergence of this integral is at most $1/2$. 
Therefore, the integral converges uniformly 
on any compact subset of the right half-plane $\Re(s) > 1/2$, 
and hence defines an analytic function there. 
This implies, by \eqref{EQ_304}, that $(\zeta'/\zeta)(s)$ has no poles 
in the right half-plane $\Re(s) > 1/2$, 
which in turn implies that the RH holds.
\hfill $\Box$

\subsection{Proof of Corollary \ref{cor_1}}

Assuming the RH, 
it is shown that $-g_0(t)$ is nonnegative for all sufficiently large $t>0$ 
by \eqref{EQ_109} and \eqref{EQ_301} as in the proof of Theorem \ref{thm_1}. 
Therefore, 
we prove the converse. 
Using \eqref{EQ_109}, \eqref{EQ_303}, and the integral formula 
\[
\int_{0}^{\infty} 4(e^{t/2}+e^{-t/2}-2)\, e^{-t(s-1/2)} \, dt 
= \frac{1}{(s-1/2)^2}\,\left(\frac{1}{s-1}+\frac{1}{s}\right) \quad \text{for}~\Re(s)>1, 
\]
we have 
\begin{equation} \label{EQ_305}
\int_{0}^{\infty} (-g_0(t)) \, e^{-t(s-1/2)} \, dt 
= \frac{1}{(s-1/2)^2} 
\left( \frac{\zeta'}{\zeta}(s)  
+ \frac{1}{s-1}  + \frac{1}{s}  \right) 
\end{equation} 
for $\Re(s)>1$. The simple pole of 
$(\zeta'/\zeta)(s)$ at $s=1$ 
is canceled out by the second term of the right-hand side. 
Hence, if we use \eqref{EQ_305} instead of \eqref{EQ_304}, 
the RH follows from the nonnegativity of $-g_0(t)$ for all sufficiently large $t>0$ 
as in the proof of Theorem \ref{thm_1}. 
\hfill $\Box$

\subsection{Proof of Theorem \ref{thm_2}} 

First, we prove that the nonpositivity \eqref{EQ_110} 
for all sufficiently large $x>0$ implies the RH. 
Using the integral formula 
\begin{equation*} 
\int_{1}^{\infty}\frac{1}{\sqrt{n}} \cdot \mathbf{1}_{[n,\infty)}(x)\cdot x^{-s+1/2} \, \frac{dx}{x} 
= \frac{1}{s-1/2} \, n^{-s} \quad \text{for}~\Re(s)>1/2  
\end{equation*}
obtained by a straightforward calculation for any $n \in \Z_{>0}$, 
we have
\begin{equation*} 
\aligned 
\int_{1}^{\infty}  
\left(\sum_{n \leq x} \frac{\Lambda(n)}{\sqrt{n}}\right) x^{-s+1/2} \, \frac{dx}{x} 
& = 
-\frac{1}{s-1/2}\frac{\zeta'}{\zeta}(s)  
\endaligned 
\end{equation*} 
for $\Re(s)>1$.  
Combining this with \eqref{EQ_303}
 and denoting $2$ as $\log \exp(2)$, we obtain
\begin{equation} \label{EQ_306}
\int_{1}^{\infty}  
\left(-\sum_{n \leq x} \frac{\Lambda(n)}{\sqrt{n}}\log\frac{x}{ne^2}
\right) x^{-s+1/2} \, \frac{dx}{x} 
= -
\frac{2(s-1)}{(s-1/2)^2} \frac{\zeta'}{\zeta}(s) 
\quad \text{for}~\Re(s)>1.
\end{equation}
On the right-hand side, 
as stated in the proof of Theorem \ref{thm_1}, 
$(\zeta'/\zeta)(s)$ has no poles in the half-line $(1/2,\infty)$ 
except for the simple pole at $s=1$, 
but it is canceled out by $s-1$. 
Hence, it is shown that the nonpositivity \eqref{EQ_110} 
for sufficiently large $x>0$ 
implies the RH as in the proof of Theorem \ref{thm_1}. 
\medskip

We prove the second half of the theorem. 
Applying \eqref{EQ_206} to $s=1/2$, 
\begin{equation} \label{EQ_307}
2\sqrt{x} 
=
\psi_{1/2}(x) 
+ \sum_{\rho} \frac{x^{\rho-1/2}}{\rho-1/2}  
+ \frac{\zeta'}{\zeta}\left(\frac{1}{2}\right)
- \sum_{k=1}^{\infty} \frac{x^{-2k-1/2}}{2k+1/2}
\end{equation}
for $x>1$. Substituting \eqref{EQ_307} into \eqref{EQ_301}, 
\begin{equation} \label{EQ_308}
\aligned 
\sum_{n \leq x} \frac{\Lambda(n)}{\sqrt{n}} \log\frac{x}{n} 
-2\sideset{}{'}\sum_{n \leq x} \frac{\Lambda(n)}{\sqrt{n}} 
& =
-\frac{\zeta'}{\zeta}\left(\frac{1}{2}\right) \log x \\
& \quad 
+ 2\sum_{\rho} \frac{x^{\rho-1/2}}{\rho-1/2}
- \sum_\rho \frac{x^{\rho-1/2}}{(\rho-1/2)^2} \\
& \quad 
- \left(\frac{\zeta'}{\zeta}\right)^\prime\left(\frac{1}{2}\right) 
+ 2\frac{\zeta'}{\zeta}\left(\frac{1}{2}\right) \\
& \quad 
- \sum_{k=1}^\infty \frac{x^{-2k-1/2}}{(2k+1/2)^2} 
- 2\sum_{k=1}^{\infty} \frac{x^{-2k-1/2}}{2k+1/2}. 
\endaligned 
\end{equation}
By denoting $2$ as $\log \exp(2)$, 
the left-hand side of \eqref{EQ_308} is written as
\[
\sum_{n \leq x} \frac{\Lambda(n)}{\sqrt{n}} \log\frac{x}{n e^2} 
+O\left(\frac{1}{\sqrt{x}}\log x\right), 
\]
where the error term arises by replacing $\sideset{}{'}\sum_{n \leq x}$ with $\sum_{n \leq x}$. 
Substituting this into \eqref{EQ_308}, 
and then replacing $x$ with $xe^2$ and dividing by $\log x$, we have 
\begin{equation} \label{EQ_309}
\aligned 
\sum_{n \leq xe^2} \frac{\Lambda(n)}{\sqrt{n}}\left(1-\frac{\log n}{\log x} \right)
& =
-\frac{\zeta'}{\zeta}\left(\frac{1}{2}\right) 
+ \frac{2}{\log x} \sum_{\rho} \frac{(xe^2)^{\rho-1/2}}{\rho-1/2} \\
& \quad - \frac{1}{\log x} \sum_\rho \frac{(xe^2)^{\rho-1/2}}{(\rho-1/2)^2} 
+O\left(\frac{1}{\log x}\right)
\endaligned 
\end{equation}
as $x \to \infty$. 

Suppose that equation \eqref{EQ_111} holds. 
Then, inequality \eqref{EQ_110} is valid for all sufficiently large $x>0$ by \eqref{EQ_302}, 
and therefore, the RH holds from the first half of the theorem. 
From the RH, the third term on the right hand side of \eqref{EQ_309} is $O(1/\log x)$, 
and  
\[
\sum_\rho \frac{x^{\rho-1/2}}{\rho-1/2} - \sum_\rho \frac{x^{\rho-1/2}}{\rho}
= \frac{1}{2}\sum_\rho \frac{x^{\rho-1/2}}{\rho(\rho-1/2)} =O(1). 
\]
Hence, \eqref{EQ_111} implies that \eqref{EQ_112} holds. 
Conversely, assuming that the RH and \eqref{EQ_112} hold, 
\eqref{EQ_111} follows from \eqref{EQ_309}. 
\hfill $\Box$

\section{Proof of Theorems \ref{thm_3}, \ref{thm_4}, and \ref{thm_5}} \label{section_4}

\subsection{Proof of Theorem \ref{thm_3}}

We prove the following theorem 
as Theorem \ref{thm_3} is obtained by applying it to $\chi=\chi_4$. 

\begin{theorem} \label{thm_8}
Let $\chi$ be a non-principal Dirichlet character of modulo $q$,   
and denote $L(s,\chi)$ by $L(s)$.
Suppose that there exists $x_0 \geq 2$ such that 
\begin{equation} \label{EQ_401}
\sum_{n \leq x} \frac{\Lambda(n)}{\sqrt{n}}\Re(\chi(n))\log\frac{x}{n} 
\end{equation}
has a constant sign for all $x \geq x_0$, 
and $L(\sigma)\not=0$ for $\sigma>\beta$ 
for some $1/2 \leq \beta <1$. 
Then, $L(s)\not=0$ in the right-half plane $\Re(s)>\beta$. 
In particular, if \eqref{EQ_401} has a constant sign for all sufficiently large $x>0$,  
and $L(\sigma)\not=0$ for $\sigma>1/2$, 
then the GRH for $L(s)$ holds. 
 
Suppose that $L(1/2)\not=0$. Then, we have 
\begin{equation} \label{EQ_402} 
\lim_{x \to \infty}
\sum_{n \leq x} \frac{\Lambda(n)\chi(n)}{\sqrt{n}}
\left(1- \frac{\log n}{\log x}\right) = -\frac{L'}{L}\left(\frac{1}{2}\right)
\end{equation}
if and only if the GRH for $L(s)$ holds.  
If $\chi$ is primitive, then 
equation \eqref{EQ_402} implies that 
the sum \eqref{EQ_401} has a constant sign 
for all sufficiently large $x>0$. 

Furthermore, suppose that $s=1/2$ is a zero of $L(s)$ of order $m \geq 0$. 
Then, we have 
\begin{equation} \label{eq_0612_1} 
\lim_{x \to \infty} \frac{1}{\log x}
\sum_{n \leq x} \frac{\Lambda(n)\chi(n)}{\sqrt{n}}
\left(1- \frac{\log n}{\log x}\right) = -\frac{m}{2}
\end{equation}
if and only if the GRH for $L(s)$ holds. 
Equation \eqref{eq_0612_1} implies that 
the value of the sum \eqref{EQ_401} is negative 
for all sufficiently large $x>0$ if $m \geq 1$. 
\end{theorem}

\noindent
{\bf Remark.} 
It is conjectured that $L(\sigma,\chi) \ne 0$ 
for all $1/2 \le \sigma \le 1$ and for all $\chi$. 
This conjecture has been confirmed for primitive Dirichlet characters of modulus $q \le 4 \cdot 10^5$ 
by Platt~\cite{Pl16}, and for real odd Dirichlet characters of modulus $q \le 3 \cdot 10^8$ 
by Watkins~\cite{Wa04}. 
It has also been confirmed by Platt~\cite{Pl16} that $L(1/2,\chi) \ne 0$ 
for primitive Dirichlet characters of modulus $q \le 2 \cdot 10^6$.

\begin{proof} 
We provide only a sketch of the proof, as the argument is similar to that of 
Theorems \ref{thm_1} and \ref{thm_2}. 
However, we carefully address the points that require particular consideration 
in the case of $L(s)$. 
We set
\begin{equation} \label{EQ_403} 
f_\chi(x):=\sum_{n \leq x} \frac{\Lambda(n)\chi(n)}{\sqrt{n}}\log\frac{x}{n}. 
\end{equation}
Then, the integral formula 
\begin{equation} \label{EQ_404}
\int_{1}^{\infty}  
(-f_\chi(x)) \, x^{-s+1/2} \, \frac{dx}{x} 
= \frac{1}{(s-1/2)^2} \frac{L^\prime}{L}(s) 
\quad \text{for}~\Re(s)>1
\end{equation} 
is obtained in the same way as \eqref{EQ_303} 
using the series expansion \eqref{EQ_209}. 
On the right-hand side, there exists $1/2 \leq \beta < 1$ such that 
$L(\sigma)\not=0$ for any real number $\sigma>\beta$,  
because $L(s)$ is analytic in $\C$,  
and has no poles and zeros in the half-line $[1,\infty) \subset \R$ 
by the Euler product \eqref{EQ_208} and Dirichlet's theorem \eqref{EQ_207}. 
Therefore, if $\chi$ is real, the first half of the theorem 
follows similarly to the proof of Theorem \ref{thm_1}. 

We consider the case where \( \chi \) is not real.  
We replace \( \chi \) by \( \bar{\chi} \) in \eqref{EQ_404}.  
Combining this with \eqref{EQ_404}, we have 
\begin{equation} \label{eq_0611_1}
\aligned 
\int_{1}^{\infty}  
\left( 
-\sum_{n \leq x} \frac{\Lambda(n)}{\sqrt{n}} \Re(\chi(n))\log\frac{x}{n}
\right) & x^{-s+1/2} \, \frac{dx}{x} \\
& 
= \frac{1}{2(s - 1/2)^2} 
\left( 
\frac{L'}{L}(s, \chi) + \frac{L'}{L}(s, \bar{\chi}) 
\right).
\endaligned 
\end{equation}
By $L(s, \bar{\chi}) = \overline{L(\bar{s}, \chi)}$,  
the assumption that \( L(\sigma, \chi) \ne 0 \) for \( \sigma > \beta \)  
implies \( L(\sigma, \bar{\chi}) \ne 0 \) for \( \sigma > \beta \).  
This yields that the right-hand side of \eqref{eq_0611_1}  
has no poles on the interval \( (\beta, \infty) \subset \R\).  
Thus, the assumption on \eqref{EQ_401} implies 
that the abscissa of convergence for the integral on the left-hand side of 
\eqref{eq_0611_1} is at most $\beta$ thanks to Proposition~\ref{prop_1}.  
This means that the right-hand side of \eqref{eq_0611_1}
is analytic in $\Re(s) > \beta$.  
This forces $L(s, \chi) \ne 0$ in $\Re(s) > \beta$, 
since the residues of the logarithmic derivative at the zeros are positive integers and hence cannot cancel between  $(L'/L)(s,\chi)$ and $(L'/L)(s,\bar{\chi})$. 
\medskip

We prove the second half of the theorem. 
Let $\chi^\ast$ be a primitive character that induces $\chi$. 
Then, by applying the argument in the proof of \cite[Lemma 1]{So09} 
to $L(s)$, 
we obtain 
\begin{equation} \label{eq_0612_3} 
\aligned 
\sum_{n \leq x} \frac{\Lambda(n)\chi(n)}{n^s}\log\frac{x}{n}
& =
-\frac{L'}{L}(s) \log x 
- \sum_\rho \frac{x^{\rho-s}}{(\rho-s)^2} 
\\
& \quad 
- \left(\frac{L'}{L}\right)^\prime(s)
- \sum_{k=0}^{\infty} \frac{x^{-2k-\kappa^\ast-s}}{(2k+\kappa^\ast+s)^2}
\endaligned 
\end{equation}
for $x>1$, $s\not=\rho_\chi, -2k-\kappa^\ast$ ($k \in \Z_{\geq 0}$). 
If $L(1/2) \not= 0$, that is, if $m = 0$, then substituting $s = 1/2$ into \eqref{eq_0612_3} yields
\begin{equation} \label{EQ_405} 
\aligned 
f_\chi(x)
& =
-\frac{L'}{L}\left(\frac{1}{2}\right) \log x 
- \sum_\rho \frac{x^{\rho-1/2}}{(\rho-1/2)^2} 
\\
& \quad 
- \left(\frac{L'}{L}\right)^\prime\left(\frac{1}{2}\right) 
- \sum_{k=0}^{\infty} \frac{x^{-2k-\kappa^\ast-1/2}}{(2k+\kappa^\ast+1/2)^2},  
\endaligned 
\end{equation}
where the sum $\sum_\rho$ is taken over all nontrivial zeros of $L(s,\chi^\ast)$,  
which coincide with all zeros of  $\xi(s,\chi^\ast)$, 
and the zeros of the finite product $L(s)/L(s,\chi^\ast)=\prod_{p\mid q}(1-\chi^\ast(p)p^{-s})$, 
and $\kappa^\ast=\kappa(\chi^\ast)$. 

Suppose that $L(s)$ has a zero of order $m \geq 1$ at $s = 1/2$. 
Then, we have
\begin{equation} \label{eq_0612_4}
\aligned 
\lim_{s \to 1/2} & \left(
-\frac{L'}{L}(s) \log x 
- \frac{mx^{-s+1/2}}{(-s+1/2)^2} 
- \left(\frac{L'}{L}\right)^\prime(s) 
\right) \\
& = -\frac{m}{2}(\log x)^2
-\lim_{s \to 1/2} \left[\frac{L^\prime}{L}(s,\chi) - \frac{m}{s-1/2}\right] \, \log x \\
& \quad 
-\lim_{s \to 1/2} \left[\left(\frac{L^\prime}{L}\right)^\prime(s) + \frac{m}{(s-1/2)^2}\right] .
\endaligned 
\end{equation}
Therefore, 
by isolating the term of $\rho = 1/2$ in the second sum on the right-hand side of \eqref{eq_0612_3}, 
then taking the limit $s \to 1/2$, we obtain 
\begin{equation} \label{eq_0612_5}
\aligned 
f_\chi(x)
& = -\frac{m}{2}(\log x)^2
-\lim_{s \to 1/2} \left[\frac{L^\prime}{L}(s,\chi) - \frac{m}{s-1/2}\right] \, \log x \\
& \quad 
-\lim_{s \to 1/2} \left[\left(\frac{L^\prime}{L}\right)^\prime(s) + \frac{m}{(s-1/2)^2}\right] \\
& \quad 
- \sum_{\rho\not=1/2} \frac{x^{\rho-1/2}}{(\rho-1/2)^2} 
- \sum_{k=0}^{\infty} \frac{x^{-2k-\kappa^\ast-1/2}}{(2k+\kappa^\ast+1/2)^2}.  
\endaligned 
\end{equation}

Assuming the GRH for $L(s)$, 
each term of the right-hand side \eqref{EQ_405} 
is bounded except for the first term, 
and each term of the right-hand side \eqref{eq_0612_5} 
is bounded except for the first and second terms. 
Hence, \eqref{EQ_402} and \eqref{eq_0612_1} hold.  

Conversely, 
assuming \eqref{EQ_402} (resp. \eqref{eq_0612_1}), we have 
$f_\chi(x) \ll \log x$ (resp. $\ll (\log x)^2$) as $x \to \infty$. 
Therefore, the integral on the left-hand side of  \eqref{EQ_404} 
converges uniformly on any compact subset in $\Re(s)>1/2$. 
Thus, it defines an analytic function in $\Re(s)>1/2$, 
and hence, 
$L(s)$ has no zeros (or poles) in $\Re(s)>1/2$, 
that is the GRH for $L(s)$. 
\medskip

Finally, we prove that the sum \eqref{EQ_401} has a constant sign 
for all sufficiently large $x>0$ assuming $L(1/2) \ne 0$, \eqref{EQ_402} 
and that $\chi$ is primitive. 
Taking the logarithmic derivative of \eqref{EQ_211} 
and substituting $s=1/2$, we obtain $\Re(\xi'/\xi)(1/2,\chi)=0$. 
This gives 
\begin{equation} \label{EQ_406}
\Re\frac{L'}{L}\left(\frac{1}{2}\right)
=\frac{1}{2}
\left[ \log\frac{\pi}{q}-\frac{\Gamma'}{\Gamma}\left(\frac{1}{4}+\frac{\kappa(\chi)}{2}\right) 
\right]
\end{equation}
by \eqref{EQ_210}. 
In particular, $\Re(L'/L)(1/2)$ depends only on $q$ and the parity of $\chi$. 
By \eqref{EQ_406}, 
the value $\Re(L'/L)(1/2)$ vanishes only when $q>1$ takes the following values
\[
q = \pi \cdot \exp\left[-\frac{\Gamma'}{\Gamma}\left( \frac{1}{4} + \frac{\kappa(\chi)}{2} \right)\right]
= 
\begin{cases}
~215.332\cdots, & \kappa=0, \\
~9.305\cdots, & \kappa=1,
\end{cases}
\]
however it is impossible for any integer $q>1$, where 
\begin{equation} \label{eq_0613_1}
\frac{\Gamma'}{\Gamma}\left( \frac{1}{4} + \frac{\kappa(\chi)}{2} \right)
= \left( \kappa(\chi) - \frac{1}{2} \right) \pi - 3\log 2 -C_0
\end{equation}
(\cite[p. 411]{MV07}). 
Therefore, $\Re(L'/L)(1/2)$ is a nonzero real number. 
Hence, by taking the real part of \eqref{EQ_402}, 
it follows that \eqref{EQ_401} has a constant sign for all sufficiently large $x>0$.  
\end{proof}

\noindent
{\bf Proof of Theorem \ref{thm_3}}. 
From the second half of Theorem \ref{thm_8}, 
we see that the first and third claims of Theorem \ref{thm_3} are equivalent. 
Using the series expansion $\lim_{N \to \infty}\sum_{n=0}^{N}(-1)^{n}(2n+1)^{-s}$ 
for $\Re(s)>0$, 
we find that $L(s, \chi_4)$ has no zeros on the positive real line $(0,\infty)$.   
Therefore, by the first half of Theorem \ref{thm_8}, 
the second claim of Theorem \ref{thm_3} implies the first claim. 
Assuming the first claim of Theorem \ref{thm_3}, 
it follows that the third claim holds by the equivalence we have already shown. 
Furthermore, since 
$\chi_4$ is primitive, the third claim implies the second claim.
\hfill $\Box$

\subsection{Proof of Theorem \ref{thm_4}}

We prove the following theorem 
as Theorem \ref{thm_4} is obtained by applying it to $\chi=\chi_4$. 
Note that $L(1/2,\chi_4) \ne 0$ as mentioned in the proof of Theorem \ref{thm_3}. 

\begin{theorem} 
Let $\chi$ be a real Dirichlet character such that 
$L(\sigma,\chi)\not=0$ for $\sigma>\beta~(\geq 1/2)$.  
Suppose that 
\begin{equation} \label{EQ_407}
\sum_{p \leq x} \chi(p) \log p \cdot \sqrt{\frac{x}{p}}\,\log\frac{x}{p}  \leq 0
\end{equation}
for all sufficiently large $x>1$. 
Then, $L(s,\chi) \not=0$ in the right-half plane $\Re(s)>\beta$. 
In particular, if $\beta=1/2$, the GRH for $L(s,\chi)$ holds. 

Conversely, suppose that the GRH for $L(s,\chi)$ holds and $L(1/2,\chi) \ne 0$. 
Then, the asymptotic formula 
\begin{equation} \label{EQ_408}
\sum_{p \leq x} \chi(p) \log p \cdot \sqrt{\frac{x}{p}}\,\log\frac{x}{p}  = -\frac{\sqrt{x}}{4}(\log x)^2
\left(1+O\left(\frac{1}{\log x}\right) \right)
\end{equation}
holds as $x \to \infty$. Therefore, 
\eqref{EQ_407} is valid for all sufficiently large $x>1$, 
and the left-hand side of \eqref{EQ_407} tends to $-\infty$ as $x \to \infty$. 

\end{theorem}

\begin{proof} Let $f_\chi(x)$ be the function in \eqref{EQ_403}. 
If we define 
\[
f_1(x):= \sum_{p \leq x} \frac{\chi(p)\log p}{\sqrt{p}} \log\frac{x}{p}, 
\]
\[
f_2(x) :=  \sum_{p \leq \sqrt{x}} \frac{\chi(p)^2\log p}{p} \log\frac{x}{p^2}, 
\qquad 
f_3(x) :=  \sum_{k \geq 3}\sum_{p \leq x^{1/k}} \frac{\chi(p)^k\log p}{p^{k/2}}\log\frac{x}{p^k}, \\
\]
then $f_\chi(x)=f_1(x)+f_2(x)+f_3(x)$. For $f_2(x)$, we have 
\[
\aligned 
f_2(x)
& = \log x \sum_{p \leq \sqrt{x}} \frac{\log p}{p} 
-2\sum_{p \leq \sqrt{x}} \frac{(\log p)^2}{p} +O(\log x)\\
& = \frac{1}{2}(\log x)^2 + O(\log x)
-2\sum_{p \leq \sqrt{x}} \frac{(\log p)^2}{p} 
\endaligned 
\]
by using \eqref{EQ_201}. 
For the third term on the right-hand side,  
\[
\aligned 
\sum_{p \leq x} \frac{(\log p)^2}{p} 
& = \log x \sum_{p \leq x}\frac{\log p}{p} 
- \int_{1}^{x} \left( \sum_{p \leq t} \frac{\log p}{p} \right)\frac{1}{t} \, dt \\
& = (\log x)^2 +O(\log x) - \int_{1}^{x} (\log t+O(1)) \frac{1}{t} \, dt \\
& = \frac{1}{2}(\log x)^2+O(\log x)
\endaligned 
\]
by partial summation and \eqref{EQ_201}. Therefore, 
\begin{equation} \label{EQ_409}
f_2(x)
= \frac{1}{4}(\log x)^2+O(\log x). 
\end{equation} 
On the other hand, 
\begin{equation} \label{EQ_410}
f_3(x)
= O(\log x)
\end{equation} 
by the convergence of the sum. As a consequence, 
if we rewrite \eqref{EQ_404} as
\begin{equation} \label{EQ_411}
\aligned 
-\int_{1}^{\infty} f_1(x) \, x^{-s+1/2} \frac{dx}{x} 
& = \frac{1}{(s-1/2)^2}\frac{L'}{L}(s) \\
& \quad + \int_{1}^{\infty} f_2(x) \, x^{-s+1/2} \frac{dx}{x} 
+ \int_{1}^{\infty} f_3(x) \, x^{-s+1/2} \frac{dx}{x}, 
\endaligned 
\end{equation} 
the second and third terms on the right-hand side 
converge absolutely and uniformly on any compact subset 
in $\Re(s)>1/2$, and hence define analytic functions there. 

Suppose that $f_1(x)$ has a constant sign for all sufficiently large $x > 1$. 
Then, the assumption $L(\sigma,\chi)\not=0$ for $\sigma>\beta~(\geq 1/2)$ 
implies that the abscissa of convergence $\sigma_c$ for the integral on the left-hand side of 
\eqref{EQ_411} satisfies $\sigma_c \leq \beta$ by Proposition~\ref{prop_1}. 
Therefore, the integral on the left-hand side of 
\eqref{EQ_411} converges uniformly for any compact subset of $\Re(s)>\sigma_c$. 
This means that the right-hand side of \eqref{EQ_411} is analytic in $\Re(s) > \beta~(\geq 1/2)$. 
This forces $L(s, \chi) \ne 0$ in $\Re(s) > \beta~(\geq 1/2)$. 

Conversely, suppose that the GRH for $L(s)$ holds. 
Then, $f_\chi(x)=O(\log x)$ by \eqref{EQ_405}. 
On the other hand, 
$f_\chi(x)=f_1(z) + (1/4)(\log x)^2 + O(\log x)$
by \eqref{EQ_409} and \eqref{EQ_410}. 
Combining these two, 
\[
f_1(x) = -\frac{1}{4}(\log x)^2\left(1+O\left(\frac{1}{\log x}\right)\right). 
\]
This immediately implies \eqref{EQ_408}. 
\end{proof}

\subsection{Proof of Theorem \ref{thm_5}}

It suffices to show  \eqref{EQ_120} with 
$x$ replaced by $x/e^2$. 
Using the decomposition 
\[
\aligned 
\sum_{n \leq x}\frac{\Lambda(n)}{\sqrt{n}}\log\frac{x}{ne^2} 
& = \sum_{p \leq x} \frac{\log p}{\sqrt{p}} \left(\log\frac{x}{p}-2\right) \\
& \quad 
+ \sum_{p \leq \sqrt{x}} \frac{\log p}{p} \left(\log\frac{x}{p^2}-2\right) \\
& \quad 
+ \sum_{k \geq 3}\sum_{p \leq x^{1/k}} \frac{\log p}{p^{k/2}}\left(\log\frac{x}{p^k}-2\right) 
\endaligned 
\]
and \eqref{EQ_201}, 
it can be obtained by an argument similar to that of Theorem $\ref{thm_4}$. 
\hfill $\Box$

\section{Proof of Theorems \ref{thm_6} and \ref{thm_7}} \label{section_5}

\subsection{Proof of Theorem \ref{thm_6}} 

Noting the orthogonality of characters
\[
\sum_{\chi\,{\rm mod}\,q} \chi(n)
= 
\begin{cases}
~\varphi(q) & n \equiv 1\,{\rm mod}\, q, \\
~0 & n \not\equiv 1\,{\rm mod}\, q,
\end{cases}
\]
we take the sum of \eqref{EQ_403} over all  $\chi$ modulo $q$. Then, 
\begin{equation} \label{EQ_501}
\sum_{\chi\,{\rm mod}\, q} f_\chi(x) 
=
\varphi(q)\sum_{{n \leq x}\atop{n \equiv 1 \,{\rm mod}\, q}} \frac{\Lambda(n)}{\sqrt{n}}\log\frac{x}{n}.
\end{equation}
The Dirichlet theorem on arithmetic progressions \cite[Corollary 11.17]{MV07} 
tells us that 
the leading term in the asymptotic formulas for both \eqref{EQ_501} 
and
\[
2\,\sum_{\chi\,{\rm mod}\, q} \sum_{n \leq x} \frac{\Lambda(n)\chi(n)}{\sqrt{n}}
=
2\,\varphi(q)\sum_{{n \leq x}\atop{n \equiv 1 \,{\rm mod}\, q}} \frac{\Lambda(n)}{\sqrt{n}}
\]
as $x \to \infty$ is $4\sqrt{x}$. Thus, leading terms cancel out in 
\begin{equation} \label{EQ_502}
\aligned 
f_q(x) 
&:= \frac{1}{\varphi(q)} \sum_{\chi\,{\rm mod}\, q} f_\chi(x) 
- \frac{4\sqrt{x}}{\varphi(q)} \\ 
& = \sum_{{n \leq x}\atop{n \equiv 1 \,{\rm mod}\, q}} \frac{\Lambda(n)}{\sqrt{n}}\log\frac{x}{n}
- \frac{4\sqrt{x}}{\varphi(q)}
\endaligned 
\end{equation}
and
\begin{equation} \label{EQ_503}
\aligned 
F_q(x) 
&:= \frac{1}{\varphi(q)} \sum_{\chi\,{\rm mod}\, q} f_\chi(x) 
- \frac{2}{\varphi(q)} \sum_{\chi\,{\rm mod}\, q} 
\sum_{n \leq x} \frac{\Lambda(n)\chi(n)}{\sqrt{n}} \\
&= \frac{1}{\varphi(q)} \sum_{\chi\,{\rm mod}\, q} \sum_{n \leq x} 
\frac{\Lambda(n)\chi(n)}{\sqrt{n}}\log\frac{x}{ne^2} \\
& = \sum_{{n \leq x}\atop{n \equiv 1 \,{\rm mod}\, q}} \frac{\Lambda(n)}{\sqrt{n}}\log\frac{x}{ne^2}. 
\endaligned 
\end{equation}
Taking the sum $\varphi(q)^{-1}\sum_{\chi\,{\rm mod}\, q}$ in \eqref{EQ_404}, 
\begin{equation} \label{EQ_504}
\aligned  
\int_{1}^{\infty}  & 
(-f_q(x)) x^{-s+1/2} \, \frac{dx}{x} 
= \frac{1}{\varphi(q)}
\left(
\frac{1}{(s-1/2)^2}  \sum_{\chi\,{\rm mod}\, q} \frac{L'}{L}(s,\chi) +\frac{4}{s-1}
\right)
\endaligned 
\end{equation} 
for $\Re(s)>1$. 
On the other hand, 
by \eqref{EQ_306}, \eqref{EQ_404}, and the way of their proofs, we obtain 
\begin{equation} \label{EQ_505}
\int_{1}^{\infty}  
\left(-\sum_{n \leq x} \frac{\Lambda(n)\chi(n)}{\sqrt{n}}\log\frac{x}{ne^2}
\right) x^{-s+1/2} \, \frac{dx}{x} 
= 
\frac{2(1-s)}{(s-1/2)^2} \frac{L'}{L}(s,\chi) 
\end{equation}
for $\Re(s)>1$. 
Taking the sum $\varphi(q)^{-1}\sum_{\chi\,{\rm mod}\, q}$ in \eqref{EQ_505}, 
\begin{equation} \label{EQ_506}
\int_{1}^{\infty}  (-F_q(x)) x^{-s+1/2} \, \frac{dx}{x} 
= \frac{2(1-s)}{(s-1/2)^2} \frac{1}{\varphi(q)} \sum_{\chi\,{\rm mod}\, q}
\frac{L'}{L}(s,\chi).
\end{equation}
Then, following an argument similar to the proof of Theorems \ref{thm_2} and \ref{thm_8}, 
and using \eqref{EQ_504} and \eqref{EQ_506}, 
we obtain the first claim of Theorem \ref{thm_6}. 

We now prove the second and third claims of Theorem \ref{thm_6}. 
Suppose that $L(1/2,\chi)\not=0$ for any Dirichlet character $\chi$ modulo $q$. 
In an argument similar to the proof of \eqref{EQ_308}, 
in the case $\chi=\chi_0$, 
we include the zeros of $\prod_{p|q}(1-p^{-s})$ 
in both sums for $\rho$ on the right-hand side of \eqref{EQ_308};  
in the case $\chi\neq \chi_0$, 
we use \eqref{EQ_405} and \eqref{EQ_213} instead of 
\eqref{EQ_301} and \eqref{EQ_206}, respectively, 
and include the zeros of $\prod_{p|q}(1-\chi^\ast(p)p^{-s})$ 
in both sums for $\rho$ on the right-hand side, 
to obtain
\begin{equation} \label{EQ_507}
\aligned 
\sum_{n \leq x}\frac{\Lambda(n)\chi(n)}{\sqrt{n}}\log\frac{x}{ne^2}  
& = 
-\frac{L'}{L}\left(\frac{1}{2},\chi\right) \log x 
+ 2\sum_{\rho_\chi} \frac{x^{\rho_\chi-1/2}}{\rho_\chi-1/2} \\
& \quad 
- \sum_{\rho_\chi} \frac{x^{\rho_\chi-1/2}}{(\rho_\chi-1/2)^2} 
- \left(\frac{L'}{L}\right)^\prime\left(\frac{1}{2},\chi\right)
+ 2\frac{L'}{L}\left(\frac{1}{2},\chi\right) \\
& \quad 
- \sum_{k=0}^{\infty} \frac{x^{-2k-\kappa^\ast-1/2}}{(2k+\kappa^\ast+1/2)^2} 
- 2\sum_{k=0}^{\infty} \frac{x^{-2k-\kappa^\ast-1/2}}{2k+\kappa^\ast+1/2} \\
& \quad 
+O\left(\frac{\log x}{\sqrt{x}}\right). 
\endaligned 
\end{equation}
Here, the two sums $\sum_{\rho_\chi}$ on the right-hand side 
are taken over all nontrivial zeros of $L(s,\chi^\ast)$ 
and all zeros of the finite product $L(s,\chi)/L(s,\chi^\ast)=\prod_{p\mid q}(1-\chi^\ast(p)p^{-s})$, 
$\kappa^\ast=\kappa(\chi^\ast)$, and  
the last term arises from the difference between $\sideset{}{'}\sum_{n \leq x}$ and 
$\sum_{n \leq x}$. 
Taking the sum $\sum_{\chi\,{\rm mod}\, q}$ in \eqref{EQ_405} 
and using \eqref{EQ_301} for $\chi=\chi_0$, we obtain 
\begin{equation} \label{EQ_508}
\aligned 
\varphi(q) f_q(x)
& =
- \log x \sum_{\chi\,{\rm mod}\, q} \frac{L'}{L}\left(\frac{1}{2},\chi\right)
-
\sum_{\chi\,{\rm mod}\, q} 
\sum_{\rho_{\chi^\ast}} \frac{x^{\rho_{\chi^\ast}-1/2}}{(\rho_{\chi^\ast}-1/2)^2} 
+O(1), 
\endaligned 
\end{equation}
where 
the sums $\sum_{\rho_{\chi^\ast}}$ on the right-hand side 
are taken over all nontrivial zeros of $L(s,\chi^\ast)$ 
and 
the implied constant depends only on $q$. 
Similarly, taking the sum $\sum_{\chi\,{\rm mod}\, q}$ in \eqref{EQ_507}, we obtain
\begin{equation} \label{EQ_509}
\aligned 
\varphi(q)F_q(x)
& =
- \log x  \sum_{\chi\,{\rm mod}\, q} \frac{L'}{L}\left(\frac{1}{2},\chi\right)
\\
& \quad 
+
\sum_{\chi\,{\rm mod}\, q}
\left[
 2\sum_{\rho_{\chi^\ast}} \frac{x^{\rho_{\chi^\ast}-1/2}}{\rho_{\chi^\ast}-1/2} 
- \sum_{\rho_{\chi^\ast}} \frac{x^{\rho_{\chi^\ast}-1/2}}{(\rho_{\chi^\ast}-1/2)^2} 
\right]
+O(1), 
\endaligned 
\end{equation}
where 
the sums $\sum_{\rho_\chi^\ast}$ and 
the implied constant have the same meaning as in \eqref{EQ_508}. 
In \eqref{EQ_509}, the contribution from the zeros of $L(s,\chi)/L(s,\chi^\ast)$ 
is included in the $O(1)$ term, and it is justified as follows.

For a Dirichlet character $\chi$ modulo $q$, 
we denote by $\mathcal{Z}_1(\chi)$ the set of 
all nontrivial zeros of $L(s,\chi^\ast)$ counted with multiplicity 
and by $\mathcal{Z}_2(\chi)$ the set of 
all zeros of $L(s,\chi)/L(s,\chi^\ast)$. 
Then, the second sum on the right-hand side of \eqref{EQ_213} for $s=1/2$ is written as
\[
\sum_{\rho_\chi} \frac{x^{\rho_\chi-1/2}}{\rho_\chi-1/2}
= \lim_{T \to \infty} 
\left( \sum_{{\rho_\chi \in \mathcal{Z}_1(\chi)}\atop{|\Im(\rho_\chi)|\leq T}} 
\frac{x^{\rho_\chi-1/2}}{\rho_\chi-1/2}
+\sum_{{\rho_\chi \in \mathcal{Z}_2(\chi)}\atop{|\Im(\rho_\chi)|\leq T}}  
\frac{x^{\rho_\chi-1/2}}{\rho_\chi-1/2} \right).
\]
The second sum inside the brackets on the right-hand side 
is $x^{-1/2}(1+o(1))$ with respect to $T$, 
because the zeros of $p$-factors of 
$L(s, \chi)/L(s, \chi^\ast)$ 
have the form $i(A + Bn)$ ($n \in \mathbb{Z}$) 
for some real numbers $A$ and $B\not=0$, 
and the series 
\[
\lim_{N \to \infty} \sum_{|n| \leq N} 
\frac{x^{iBn}}{i(A + Bn)-1/2}
\] 
is convergent and bounded, as in the case of the Fourier series of the sawtooth wave. 
Therefore, the contribution from the zeros of $L(s,\chi)/L(s,\chi^\ast)$  
on the right-hand side of \eqref{EQ_507}, which originates from \eqref{EQ_213}, is bounded. 
Moreover, 
the contributions from the zeros of $L(s,\chi)/L(s,\chi^\ast)$ arising from \eqref{EQ_301} and \eqref{EQ_405} are also bounded, 
since the sums over zeros in \eqref{EQ_301} and \eqref{EQ_405} converge absolutely.

Assuming that the GRH holds for $L(s,\chi)$ for all $\chi$ modulo $q$, 
the second sum on the right-hand side of \eqref{EQ_508} is bounded. 
Additionally, if we assume \eqref{EQ_125}, 
the last sum over $\chi$ mod $q$ on the right-hand side of \eqref{EQ_509} 
is $o(\log x)$. 
Hence, \eqref{EQ_123} and \eqref{EQ_124} hold. 

Conversely, 
assuming \eqref{EQ_123} (resp. \eqref{EQ_124}), 
it follows from \eqref{EQ_502} (resp. \eqref{EQ_503}) that 
$f_q(x) \ll \log x$ (resp. $F_q(x) \ll \log x$) as $x \to \infty$. 
Therefore, the integral on the left-hand side of \eqref{EQ_504} (resp. \eqref{EQ_506}) converges uniformly 
on any compact subset of the right half-plane $\Re(s) > 1/2$, and hence defines an analytic function there. 
Thus, the GRH for $L(s,\chi)$ for all $\chi$ modulo $q$ follows 
from an argument similar to that in the proof of Theorem~\ref{thm_8}. 
Moreover, assuming \eqref{EQ_124}, it follows from \eqref{EQ_503} that
\[
\varphi(q) F_q(x) 
= 
- \log x  \sum_{\chi\,{\rm mod}\, q} \frac{L'}{L}\left(\frac{1}{2},\chi\right) + o(\log x)
\quad \text{as } x \to \infty.
\]
Comparing this with \eqref{EQ_509}, we find that the sum over $\chi$ mod $q$ on the right-hand side of \eqref{EQ_509} must be $o(\log x)$. However, the latter part of this sum is absolutely convergent and bounded, and thus \eqref{EQ_125} follows.

\comment{
In the right-hand side of \eqref{EQ_504} (resp.\ \eqref{EQ_506}),  
the residues of $(L'/L)(s,\chi)$ at its poles are equal to the orders of the zeros of $L(s,\chi)$, and are therefore positive integers,  
except for the pole of $(L'/L)(s,\chi_0)$ at $s=1$, which is canceled in the right-hand side of \eqref{EQ_504} (resp.\ \eqref{EQ_506}).  
This implies that the poles of $(L'/L)(s,\chi)$ cannot cancel each other  in the right-hand side of \eqref{EQ_504} (resp.\ \eqref{EQ_506}).  
Hence, if the right-hand side of \eqref{EQ_504} (resp.\ \eqref{EQ_506}) is analytic in the region $\Re(s) > 1/2$,  
it follows that $L(s,\chi) \ne 0$ for all $\chi$ modulo $q$ in that region. 
Hence the GRH holds for $L(s,\chi)$ for all $\chi$ modulo $q$. 
}
\medskip

Finally, we prove the fourth and fifth claims of Theorem~\ref{thm_6}. 
By using \eqref{eq_0612_2} in place of \eqref{EQ_213} and \eqref{eq_0612_5} in place of \eqref{EQ_405}, 
and employing the resulting modified versions of \eqref{EQ_508} and \eqref{EQ_509}, 
one can show that both \eqref{eq_0612_6} and \eqref{eq_0612_7} are equivalent to the GRH for $L(s,\chi)$ 
for all Dirichlet characters $\chi$ modulo $q$, 
by an argument similar to that used in the proofs of  the second and third claims.  
Furthermore, by appropriately modifying \eqref{EQ_509} using \eqref{eq_0612_2} and \eqref{eq_0612_5}, 
one can also show that \eqref{eq_0612_7} implies \eqref{eq_0613_3} 
by an argument similar to that of the proof of the third claim. 
We omit the details, as there is no essential difference.
\hfill $\Box$

\subsection{Proof of Theorem \ref{thm_7}} 

First, we prove \eqref{EQ_126}. 
After multiplying $f_q(x)$ in \eqref{EQ_502} by $\varphi(q)$, 
summing over $3 \leq q \leq Q$,  
and then changing the order of summation, we obtain
\begin{equation} \label{EQ_510}
\aligned 
\sum_{3 \leq q \leq Q} 
\varphi(q) f_q(x) 
& = \sum_{3 \leq q \leq Q} 
\left( \varphi(q)\sum_{{n \leq x}\atop{n \equiv 1 \,{\rm mod}\, q}} 
\frac{\Lambda(n)}{\sqrt{n}}\log\frac{x}{n} -4 \sqrt{x}
\right)\\
& = \sum_{n \leq x} 
\frac{\Lambda(n)}{\sqrt{n}}\log\frac{x}{n} \sum_{{3 \leq q \leq Q}\atop{q \mid n-1}} \varphi(q)
- 4\sqrt{x}\,(Q-2), 
\endaligned 
\end{equation}
where we understand that the inner sum is zero if it is an empty sum. 
For any $n \leq x$, 
all the divisors of $n-1$ appear in the range $1 \leq q \leq Q$ 
by the assumption $Q \geq x$. 
For a prime number $p$, $2$ divides $p^k-1$ if and only if $p$ is odd. 
Therefore, using $\sum_{d|m}\varphi(d)=m$ on the right-hand side of \eqref{EQ_510},  
\[
\aligned 
\sum_{n \leq x} 
\frac{\Lambda(n)}{\sqrt{n}}\log\frac{x}{n} \sum_{{3 \leq q \leq Q}\atop{q \mid n-1}} \varphi(q)
& =
\sum_{{n \leq x}\atop{n \not\in 2^{\N}}} \frac{\Lambda(n)(n-3)}{\sqrt{n}}\log\frac{x}{n}
+ \log 2\sum_{2^k \leq x} \frac{(2^k-2)}{\sqrt{2^k}}\log\frac{x}{2^k} \\ 
& =
\sum_{n \leq x} \frac{\Lambda(n)(n-3)}{\sqrt{n}}\log\frac{x}{n}
+ o(\sqrt{x}\,).
\endaligned 
\]
Using \cite[(2.3)]{Su23} and the prime number theorem \eqref{EQ_202}
in \eqref{EQ_301}, 
\[
\sum_{n \leq x} \frac{\Lambda(n)}{\sqrt{n}}\log\frac{x}{n}
= O(\sqrt{x}). 
\]
Therefore, we obtain 
\[
\sum_{3 \leq q \leq Q} \varphi(q) f_q(x)
=
\sum_{n \leq x} \Lambda(n)\sqrt{n}\,\log\frac{x}{n} - 4\sqrt{x}\,Q
+ O(\sqrt{x}\,). 
\]
By applying partial summation to 
the first term on the right-hand side 
using the prime number theorem \eqref{EQ_202}, 
\[
\aligned 
\sum_{n \leq x} \Lambda(n) \cdot \sqrt{n}\,\log\frac{x}{n}
& = - \int_{1}^{x}  \sum_{n \leq u} \Lambda(n) 
\left(
-\frac{1}{\sqrt{u}} + \frac{1}{2\sqrt{u}}\log\frac{x}{u}
\right) \, du \\
& = \frac{4}{9}x\sqrt{x}\,(1+o(1)). 
\endaligned 
\]
Hence, we obtain \eqref{EQ_126}. 

Next, we prove \eqref{EQ_127}. 
After multiplying $F_q(x)$ in \eqref{EQ_503} by $\varphi(q)$, 
summing over $3 \leq q \leq Q$,  
and then changing the order of summation as in \eqref{EQ_510}, we obtain
\begin{equation} \label{EQ_511}
\sum_{3 \leq q \leq Q} \varphi(q) F_q(x)
= \sum_{n \leq x} 
\frac{\Lambda(n)}{\sqrt{n}}\log\frac{x}{ne^2} \sum_{{3 \leq q \leq Q}\atop{q \mid n-1}} \varphi(q). 
\end{equation}
Then, as before, 
\[
\sum_{3 \leq q \leq Q} \varphi(q) F_q(x)
 =
\sum_{n \leq x} \frac{\Lambda(n)(n-3)}{\sqrt{n}}\log\frac{x}{ne^2}
+ o(\sqrt{x}\,).
\]
Using \cite[(2.3)]{Su23} and the prime number theorem \eqref{EQ_202}
in \eqref{EQ_308}, 
\[
\sum_{n \leq x} \frac{\Lambda(n)}{\sqrt{n}}\log\frac{x}{ne^2}
= o(\sqrt{x}\,).
\]
Therefore, we obtain 
\[
\sum_{3 \leq q \leq Q} \varphi(q) F_q(x)
=
\sum_{n \leq x} \Lambda(n)\sqrt{n}\,\log\frac{x}{ne^2}
+ o(\sqrt{x}\,). 
\]
By applying partial summation to 
the first term on the right-hand side 
using the prime number theorem \eqref{EQ_202}, 
\[
\aligned 
\sum_{n \leq x} \Lambda(n) \cdot \sqrt{n}\,\log\frac{x}{ne^2}
& = -2\sqrt{x}\,\sum_{n \leq x} \Lambda(n) - \int_{1}^{x}  \sum_{n \leq u} \Lambda(n) 
\left(
-\frac{2}{\sqrt{u}} + \frac{1}{2\sqrt{u}}\log\frac{x}{u}
\right) \, du \\
& = -\frac{8}{9}x\sqrt{x}\,(1+o(1)). 
\endaligned 
\]
Hence, we obtain \eqref{EQ_127} by substituting $xe^2$ for $x$. \hfill 
$\Box$
\medskip

If we take the sum $\sum_{3 \leq q \leq Q}$ in \eqref{EQ_503}, we obtain 
\[
\sum_{3 \leq q \leq Q} 
F_q(x) 
=
\sum_{n \leq x} \frac{\Lambda(n)(\sigma_0(n-1)-\nu(n))}{\sqrt{n}}\log\frac{x}{ne^2},
\]
where $\sigma_0(n)$ denotes the number of divisors of $n$ 
and $\nu(n)=1$ if $n=2^k$ for some $k \in \Z_{>0}$, 
and $\nu(n)=2$ otherwise. 
In this case, the study of Titchmarsh's divisor problem 
$\sum_{n \leq x} \Lambda(n)\sigma_0(n-1)$ provides the leading term.
Using the asymptotic formula
\[
\aligned
\sum_{n \leq x} \Lambda(n)\sigma_0(n-1) 
& = 
\frac{\zeta(2)\zeta(3)}{\zeta(6)}x\log x \\
& \quad + \frac{\zeta(2)\zeta(3)}{\zeta(6)}\left[ 2 \left(C_0-\sum_p \frac{\log p}{p^2-p+1}\right) 
 -1\right] x + O_A\left( \frac{x}{(\log x)^A} \right)
\endaligned 
\]
obtained  by Bombieri, Friedlander, and Iwaniec \cite{BFI86} and Fouvry \cite{Fo85}, we find that
\[
\sum_{n \leq x} \frac{\Lambda(n)\sigma_0(n-1)}{\sqrt{n}}\log\frac{x}{ne^2} 
= -8\,\frac{\zeta(2)\zeta(3)}{\zeta(6)} \sqrt{x}\, (1+o(1)).
\]
This also shows that \eqref{EQ_122} has a constant sign on average, 
however, the sum used in the proof of Theorem \ref{thm_7} is probably easier to handle.

%

%
\bigskip 

\noindent
Masatoshi Suzuki,\\[5pt]
Department of Mathematics, \\
School of Science, \\
Institute of Science Tokyo \\
2-12-1 Ookayama, Meguro-ku, \\
Tokyo 152-8551, Japan  \\[2pt]
Email: {\tt msuzuki@math.sci.isct.ac.jp}

\end{document}